\documentclass{birkjour}
\usepackage{bm, amsmath, amssymb, amsthm} %, amsthm
\usepackage{amsmath, amssymb, amsthm}

\newcommand{\R}{{\mathbf R}}
\newcommand{\Z}{{\mathbf Z}}
\newcommand{\N}{{\mathbf N}}

\newcommand    {\e}{{\mathbf x}}
\newcommand    {\by}{{\mathbf y}}

\newcommand    {\C}{{\mathbf C}}

\newtheorem{theorem}{Theorem}[section]
\newtheorem{proposition}[theorem]{Proposition}
\newtheorem{lemma}[theorem]{Lemma}
\newtheorem{corollary}[theorem]{Corollary}
\theoremstyle{definition}
\newtheorem{definition}[theorem]{Definition}
\theoremstyle{remark}
\newtheorem{remark}[theorem]{Remark}
\numberwithin{equation}{section}
\newtheorem{example}{Example}

\renewcommand{\theequation}{\thesection.\arabic{equation}}

\begin{document}
	
	\title[Discreteness of the spectrum for Schr\"odinger operator]{
		Conditions for semi-boundedness \\and discreteness of the spectrum \\to Schr\"odinger operator and  some \\nonlinear PDEs}
	
	\author[L. Zelenko]
	{Leonid Zelenko} 
	
	\address{Department of Mathematics \\
		University of Haifa  \\
		Haifa 31905  \\
		Israel}
	\email{zelenko@math.haifa.ac.il}
	
	\begin{abstract}
	For Schr\"odinger operator $H=-\Delta+ V(\e)\cdot$, acting in the space $L_2(\R^d)\,(d\ge 3)$,  necessary and sufficient conditions for semi-boundedness and discreteness of its spectrum.are obtained without assumption that the potential $V(\e)$ is bounded below. By reduction of the problem to
	investigation of existence of regular solutions for Riccati PDE necessary conditions for discreteness of the spectrum of operator $H$ are obtained under assumption that it is bounded below.  These results are similar to ones obtained by author  in \cite{Zel} for the one-dimensional case. Furthermore, sufficient conditions for the semi-boundedness and discreteness of the spectrum of $H$ are obtained in terms of a non-increasing rearrangement, mathematical expectation and standard deviation from the latter for  positive part $V_+(\e)$ of the potential $V(\e)$ on compact domains that go to infinity, under certain restrictions for its negative part $V_-(\e)$. Choosing in an optimal way the vector field associated with  difference between the potential $V(\e)$ and its mathematical expectation on the balls that go to infinity, we  obtain a condition for semi-boundedness and discreteness of the spectrum for $H$ in terms of solutions of Neumann problem for nonhomogeneous $d/(d-1)$-Laplace equation. This type of optimization refers to a divergence constrained transportation problem.
 \end{abstract}
	
	\subjclass{Primary 47F05, 47B25,\\ 35P05; Secondary 81Q10} 	
	
	\keywords{Schr\"odinger operator,
		discreteness of the spectrum, \\ Riccati  PDE, p-Laplacian, transportation problem} 
	
	\maketitle
	
	%\tableofcontents
	
	\section{Introduction} \label{sec:introduction}
	\setcounter{equation}{0}
	
In the present paper we consider the Schr\"odinger operator $H=-\Delta+ V(\e)\cdot$,
acting in the space $L_2(\R^d)$. In physical and mathematical literature the coefficient $V(\e)$ is called {\it potential}, since it is the potential of electric field acting on electron. In what follows we assume that $d\ge 3$ and $V\in L_{\infty,loc}(\R^d)$. The last assumption ensures that the operator $H^0$, defined on $C_0^\infty(\R^d)$ by the operation $-\Delta+V(\e)\cdot$, acts in $L_2(\R^d)$.  The operator $H$ is meant as the closure of $H^0$. It is known that the boundedness of $H^0$ from below implies that $H$ is self-adjoint \cite[Theorem 2.1]{K-Sh}.
Sometimes we will impose stronger local conditions on  $V(\e)$\footnote{Here local restrictions for the potential are not so essential. It is important its behavior at infinity.}.

 In our papers \cite{Zel2}, \cite{Zel1} some sufficient conditions for discreteness of the spectrum of $H$ have been
obtained  on the base of  well known Mazya -Shubin criterion \cite{M-Sh}, formulated in Preliminaries (Section \ref{sec:prelim}).  Notice that  Mazya -Shubin result is a significant improvement of conditions obtained by  A.M. Molchanov \cite{Mol}. As we have noticed in \cite{Zel2}, \cite{Zel1}, conditions of Mazya-Shubin criterion are hardly verifiable, since they are formulated in terms of harmonic capacity of sets. In our papers we have obtained easier verifiable sufficient conditions in terms of non-increasing rearrangement of the potential $V(\e)$ with respect to Lebesgue measure on compact domains that go to infinity, its mathematical expectation and standard deviation from the latter with respect to the normalized Lebesgue measure on these domains. There we have used also perturbations of the potential, preserving discreteness of the spectrum of $H$.  It was assumed in \cite{Mol}, \cite{M-Sh} \cite{Zel2} and  \cite{Zel1}  that $V(\e)\ge 0$, but in a part of claims in \cite{Zel2} this condition was
not imposed on the potential.

 Notice that in the papers \cite{Ben-Fort}, \cite{Ben-Fort1}, \cite{Si1}, \cite{L-S-W} and \cite{GMD} some  sufficient conditions for discreteness of the spectrum for $H$ have been found without use of Mazya-Shubin result. In \cite{Ben-Fort1}, \cite{L-S-W} also the case of a unbounded below potential was studied and in \cite{GMD} also the case of a matrix-valued potential was considered. But for a scalar nonnegative potentials  $V(\e)$ our results from \cite{Zel2}, \cite{Zel1} are more general than ones mentioned in this paregraph 

  The goal of present work is to obtain new necessary and sufficient conditions for semi-boundedness and discreteness of the spectrum of operator $H$, covering potentials which are not bounded below. A distinctive feature of our work is the use of nonlinear partial differential equations for the study of this linear problem: Riccati PDE 
(Proposition \ref{prlocRicc}, proof of Theorem \ref{thnesscond1}) and nonhomogeneous $d/(d-1)$-Laplace equation, connected with a divergence constrained transportation problem (Corollary \ref{thchicpert}, Proposition \ref{prtranspr}).

In the one-dimensional case (d=1) Molchanov's criterion acquires rather simple form: if $V(x)\ge 0$, then the spectrum of $H$ is discrete if and only if for any $r>0$
\begin{equation*}
\lim_{|y|\rightarrow\infty}\int_y^{y+r}V(x)\,\mathrm{d}x=\infty.
\end{equation*}
Further in  the one-dimensional case criteria for semi-boundedness and discreteness of the spectrum of $H$, covering unbounded below potentials,   were obtained in \cite{Br}, \cite{Is}, \cite{Zel}, \cite{Bru} and \cite{Mas}. Since Ismagilov's result \cite[Theorem 1] {Is} is close to some claims of this paper, we will formulate it:

\begin{theorem}
	(i) If there exist real-valued functions $\alpha(t)$ and $\beta(t)$,
defined on $(0,r]$ such that
	\begin{equation}\label{cndalphbet}
	\int_0^r\big(\alpha^2(t)+\beta^2(t)\big)
\,dt<\infty
	\end{equation}
	and for any $y\in\R$ the function
	\begin{equation*}
	{\mathcal V}_{y}(s,t)=\int_{y+s}^{y+t}V(v) \mathrm{d}x
	\end{equation*}
satisfies the condition on the triangle $\Delta_r=\{\langle s,t\rangle:\,0\le
	s\le t\le r)\}$:
	\begin{equation}\label{condcalV}
	\forall\,\langle s,t\rangle\in \Delta_r:\quad{\mathcal V}_{y}(s,t)\ge
	\alpha(t)-\beta(s),
	\end{equation}
	then the operator $H$ is bounded below;\vskip2mm
	
	\noindent (ii) If \eqref{cndalphbet} and \eqref{condcalV}  
	 are valid, then the spectrum of $H$ is discrete if and only if the condition
	\begin{equation}\label{mesconvinfty2}
	\forall\,A>0\,:\quad\lim_{|y|\rightarrow\infty}\mathrm
	\mathrm{mes}_2\big(\{\langle s,t\rangle\in\Delta_r:\;{\mathcal
		V}_{y}(s,t)\le A\}\big)=0.
	\end{equation}
	is fulfilled.
	\end{theorem}
In our work \cite{Zel} for the one-dimensional case some necessary and sufficient conditions for semi-boundedness and discreteness of spectrum of $H$, generalizing Ismagilov's result, were obtained. There we used arguments connected with Riccati ODE $u^\prime=u^2-V(x)+\lambda$. Particularly, in such a way we have shown that condition \eqref{mesconvinfty2} is necessary for discreteness of the spectrum under assumption that operator $H$ is bounded below (without explicit restrictions for the potential) (\cite[\S 3, Theorem 1]{Zel}). This result is based on an upper estimate of Lebesgue measure of some subsets of the triangle $\Delta_r$, connected with integral inequalities of Riccati type. To this end we used the blow-up behavior of some solutions to Riccati ODE (proof of Lemma \ref{lmintineq}).  Notice that 
the claims of Theorem \ref{thnesscond1} of the present paper, yielding necessary conditions for semi-boundedness and discreteness of the spectrum of $H$, can be considered as multi-dimensional analogues of  \cite[\S 3, Theorem 1]{Zel}. The proofs of these claims are based on the same results on integral inequalities from \cite{Zel}. Since \cite{Zel} has been published in Russian and there is no English translation, we have included this material on integral inequalities into Section \ref{sec:intineqRicc}.

 We have obtained also sufficient conditions for semi-boundedness and discreteness of the spectrum of $H$ in terms of the non-increasing rearrangement, mathematical  expectation and standard deviation from the latter for positive  part $V_+(\e)$ of the potential $V(\e)$ under some restrictions for its negative part $V_-(\e)$ (Theorem \ref{thgenermolch1}). 
Furthermore, choosing in an optimal way the vector field associated with the difference between the potential $V(\e)$ and its mathematical expectation on the balls that go to infinity,, we obtain a condition for semi-boundedness and discreteness of the spectrum for $H$ in terms of solution of Neumann problem for non-homogeneous $d/(d-1)$-Laplace equation, mentioned above ( Corollary \ref{thchicpert}). 

The work is organized as follows. After this Introduction in Section \ref{sec:notation} we introduce some basic notations.  In Section \ref{sec:prelim} (Preliminaries)  we recall some notions and notations from the previous paper \cite{Zel2} and formulate the Mazya-Shubin result. In Section \ref{sec:resprevpap} we formulate some results from
\cite{Zel2}, used in this work. In Section \ref{sec:meinres} we formulate main results of this work.
In Section \ref{sec:proofmainres} we prove main results. formulating previously the localization principle from \cite{Zel2} (Proposition \ref{prloc}) and proving the localization principle on the base of Riccati PDE (Proposition \ref{prlocRicc}).  In Section \ref{sec:examples} (Examples) we construct examples which permit 
to  compare our claims with each other. Sections  \ref{sec:transprob} and \ref{sec:intineqRicc} are  appendices.  In section \ref{sec:transprob} Proposition \ref{prtranspr} on the solution of a divergence constrained transportation problem is formulated, which is used in  Corollary \ref{thchicpert}. Section \ref{sec:intineqRicc} is devoted to integral inequalities of Riccati type, mentioned above.

\section{Basic notations} \label{sec:notation}
\setcounter{equation}{0}

\noindent $\e=\langle x_1, x_2,\dots x_d\rangle$ is a vector in the space $\R^d$;
\vskip1mm 
\noindent$|\e|=\sqrt{\sum_{k=1}^d x_k^2 }$ and  $|\e|_\infty=\max_{1\le j\le d}|x_j|$ are the norms in $\R^d$;
\vskip1mm

\noindent$\Z$ is the ring of integers; $\N$ is the set of natural numbers;
\vskip1mm

\noindent$[x]$ is the integer part of a real number $x$;
\vskip1mm

\noindent $\prod_{k=1}^d S_k$ is the cartesian product of sets $S_1,\,S_2,\dots,\, S_d$;

\noindent$\Z^d=\prod_{k=1}^d \Z$;\vskip1mm
\vskip1mm

\noindent $B_r(\by)$ is the open ball in $\R^d$, centered at $\by$ and having radius  $r$; $\bar B_r(\by)$ is its closure;
\vskip1mm

\noindent $\omega_d$ is the hypersurface area of unit sphere $S^{d-1}=\partial B_1( 0)$;
\vskip1mm

\noindent $\mathrm{mes}_d$ is Lebesgue measure in $\R^d$;
\vskip1mm

\noindent $\Sigma_L(\Omega)$ is the $\sigma$-algebra of all Lebesgue measurable subsets of of a domain $\Omega\subseteq\R^d$;
\vskip1mm

\noindent $\Vert f\Vert_p\,(p\in[1,\,\infty])$ is the $L_p$-norm of a  function $f\in L_p(\Omega)\,(\Omega\subseteq\R^d)$; 
 \vskip1mm
 
\noindent We denote  by $(\cdot,\cdot)_\Omega$ and  $\Vert\cdot\Vert_\Omega$ the
 inner product and  the norm in the space $L_2(\Omega)$;
\vskip1mm

\noindent For the analogous spaces of vector-valued functions and the norms in them we shall use the same notations;
\vskip1mm

\noindent $\mathrm{supp}(f)$ is the support of a function $f:\,\Omega\rightarrow\C$ ;
\vskip1mm

\noindent $C_0^\infty(\Omega)$ is the collection of all functions from $C^\infty(\Omega)$ having compact supports  in an open domain $\Omega\subseteq\R^d$.
\vskip1mm

Some specific notations will be introduced in what follows.

\section{Preliminaries}\label{sec:prelim}
	\setcounter{equation}{0}
	
 We will recall  some notions and notations from the previous paper \cite{Zel2}. 
	
	Harmonic capacity of a compact set $F\subset\R^d$ is defined in the following manner:
	\begin{eqnarray}\label{dfWincap}
	&&\mathrm{cap}(F):=\inf\Big(\big\{\int_{\R^d}|\nabla u(\e)|^2\,
	\mathrm{d}\e\,:\;u\in C^\infty(\R^d), \;u\ge
	1\;\mathrm{on}\;F,\nonumber\\
	&&u(\e)\rightarrow 0\;\mathrm{as}\;|\e|\rightarrow\infty \big\}\Big)
	\end{eqnarray}
	It is known (\cite{Ch}, \cite{Maz}, \cite{Maz1}) that the set function ``cap'' can be extended in a suitable manner from the collection of all compact subsets of the space $\R^d$ to the $\sigma$-algebra $\Sigma_B(\R^d)$ of all Borel subsets of it
and the {\it isocapacity inequality} is valid:
\begin{equation}\label{isocapineq}
\forall\,F\in\Sigma_B(\R^d):\quad\mathrm{mes}_d(F)\le
c_d\,(\mathrm{cap}(F))^{d/(d-2)}.
\end{equation}
with $c_d=\big(d(d-2)(\mathrm{mes}_d(B_1(0)))^{2/d}\big)^{-d/(d-2)}$, which  comes as
identity if $F$ is a closed ball.

Consider in $\R^d$ an open domain $\mathcal{G}$ satisfying the conditions:

(a) $\mathcal{G}$ is bounded and star-shaped with respect to any point of an open ball $B_\rho(0)\,(\rho>0)$ contained in $\mathcal{G}$;

(b) $\mathrm{diam}(\mathcal{G})=2$.

As it was noticed in \cite{M-Sh}, condition (a) implies that $\mathcal{G}$ can be represented in the form
\begin{equation}\label{formofcalG}
\mathcal{G}=\{\e\in\R^d:\,\e=r\omega,\, |\omega|=1,\,0\le r<r(\omega)\}, 
\end{equation} 
where $r(\omega)$ is a positive Lipschitz function on the standard unit sphere $S^{d-1}\subset\R^d$. For $r>0$
and $\by\in\R^d$ denote 
\begin{equation}\label{dfcalGry}
\mathcal{G}_r(\by):=\{\e\in\R^d:\,r^{-1}\e\in\mathcal{G}\}+\{\by\}.
\end{equation}

As in \cite{M-Sh}, consider the collection
$\mathcal{N}_{\gamma}(\by,r)\;(\gamma\in(0,1))$  of all
compact sets $F\subseteq\bar{\mathcal{G}}_r(\by)$ satisfying the condition
\begin{equation}\label{defNcap}
\mathrm{cap}(F)\le\gamma\,\mathrm{cap}(\bar{\mathcal G}_r(\by)),
\end{equation}

The Mazya-Shubin result, mentioned in Introduction, is:

\begin{theorem}\label{thMazSh}\cite[Theorem 2,2]{M-Sh}  
	If $V(\e)\ge 0$, the spectrum of operator $H$ is discrete as soon as 
	for some $r_0>0$ and for any $r\in(0,r_0)$ the condition
	\begin{equation}\label{cndmolch1}
	\lim_{|\by|\rightarrow\infty}\inf_{F\in\mathcal{N}_{\gamma(r)}(\by,r)}\int_{{\mathcal G}_r(\by)\setminus
		F}V(\e)\, \mathrm{d}\e=\infty,
	\end{equation}
	is satisfied, where 
	\begin{equation}\label{cndgammar}
	\forall\,r\in(0,\,r_0):\;\gamma(r)\in(0,1)\quad\mathrm{and}\quad
	\limsup_{r\downarrow 0} r^{-2}\gamma(r)=\infty,
	\end{equation}	
\end{theorem}

 Also a necessary condition for discreteness of the spectrum was obtained in \cite{M-Sh}, which is close to sufficient one. Notice that it was proved in \cite{Mol}  that the condition
\begin{equation}\label{intVtendinf} 
\forall\; r\in (0,\,r_0]\quad\lim_{|\by|\rightarrow\infty}\int_{\mathcal{G}_r(\by)}V(\e)\,\mathrm{d}\e=\infty
\end{equation}
is necessary for discreteness of the spectrum of $H$ under assumption that $V(\e)\ge 0$.

\begin{remark}\label{remsembound}
In Section \ref{sec:examples} we will construct an example (Example \ref{ex2}), which shows that condition \eqref{cndmolch1}  and semi-boundedness of the
	operator $H$ do not imply the discreteness of its spectrum, if boundedness from below of the potential $V(\e)$ is not supposed.
\end{remark}

For a function $W\in L_1(\Omega)$, where $\Omega\subset\R^d$ is a bounded domain,  consider the non-increasing rearrangement $W^\star(t,\Omega)$ of it, defined in the following manner:
\begin{equation}\label{dfSVyrdel}
W^\star(t,\Omega):=\sup\{s>0\,:\;\lambda^\star(s,\,W,\,\Omega)\ge
t\}\quad (t>0), 
\end{equation}
where
\begin{eqnarray}\label{dfLVsry}
&&\lambda^\star(s,\,W,\,\Omega)=\mathrm{mes}_d(\mathcal{L}^\star(s,\,W,\,\Omega)),\nonumber\\
&&\mathcal{L}^\star(s,\,W,\,\Omega)=\{\e\in\Omega:\; W(\e)\ge s\}.
\end{eqnarray}
In the case where $\Omega=\mathcal{G}_r(\by)$  we will write for brevity $W^\star(t,\,\by,\,r)$  instead of $W^\star(t,\mathcal{G}_r(\by))$.

Let  $(X,\,\Sigma,\,\mu)$ be a probability space.  Consider the mathematical expectation $\mathrm{E}(W)$ of a real-valued function $W(x)$: 
\begin{equation}\label{dfedxpec}
\mathrm{E}(W)=\int_XW(x)\mu(dx)
\end{equation}
and (standard) deviation from it
\begin{eqnarray}\label{dfVar}
&&\mathrm{Dev}(W)=\sqrt{E\big((W-\mathrm{E}(W))^2\big)}=\sqrt{\int_X\big(W(x)-\mathrm{E}(W)\big)^2\mu(dx)}=\nonumber\\
&&\sqrt{\mathrm{E}(W^2)-\big(\mathrm{E}(W)\big)^2}, 
\end{eqnarray}
assuming that  $W\in L_2(X,\mu)$. In particular, we will consider the probability space
\begin{equation*}
 \big(\mathcal{G}_r(\by),\Sigma_L(\mathcal{G}_r(\by)),m_{d,r}\big),
\end{equation*}
where $m_{d,r}$ is the normalized Lebesgue measure
\begin{equation}\label{dfmdr}
m_{d,r}(A):=\frac{\mathrm{mes}_d(A)}{\mathrm{mes}_d({\mathcal G}_r(\by))}\quad(A\in\Sigma_L(\bar {\mathcal G}_r(\by))).
\end{equation}
Denote by $\mathrm{E}_{\by,r}(W)$ and $\mathrm{Dev}_{\by,r}(W)$ the mathematical expectation and deviation from it for a real-valued  function $W\in L_2(\mathcal{G}_r(\by),\, m_{d,r})$

Let $\Omega$ be an open and bounded domain in $\R^d.$ Following to \cite{M-V}, for $W\in L_1(\Omega)$ and $s\ge 1$ consider the
set ${\mathcal D}_s(W,\,\Omega)$ of all vector fields $\vec F(\e)$ on $\Omega$ satisfying the divergence equation 
\begin{equation}\label{divereq}
\mathrm{div}\,\vec F=W
\end{equation} 
and belonging to the space $L_s(\Omega)$.
We assume that  this equation  is
satisfied in the sense of distributions and we have shown that ${\mathcal D}_s(W,\,\Omega)\neq\emptyset$ \cite[Definitions A.2, A.3, Propositions A.1, A.2 ]{Zel2}.  
Denote
\begin{equation}\label{defDsW}
D_s(W,\Omega)=\inf_{\vec F\in{\mathcal D}_s(W,\,\Omega)}\Vert\vec F\Vert_s,
\end{equation}
In the case where $\Omega=\mathcal{G}_r(\by)$ we will write briefly 
$\mathcal{D}_s(W,\by,r)$ and $D_s(W,\by,r)$.

\section{Some results from the previous work}\label{sec:resprevpap}
\setcounter{equation}{0}

	In  the next three claims  we assumed that $V(\e)\ge 0$.
	
	Denote by $\mathcal{M}_{\gamma}(\by,r)\;(\gamma\in(0,1))$ the
	collection of all Borel sets $F\subseteq\mathcal{G}_r(\by)$
	satisfying the condition
	$\mathrm{mes}_d(F)\le\gamma\,\mathrm{mes}_d(\mathcal{G}_r(\by))$.
	A direct use of Theorem \ref{thMazSh} and  isocapacity inequality \eqref{isocapineq}
	leads to the following claim:
	
	\begin{proposition}\label{thcondlebesgue}\cite[ Proposition 4.5]{Zel2}
		If for some $r_0>0$ and a  function $\gamma(r)$,	satisfying the conditions 
		\begin{equation}\label{cndtildgam1}
		\forall\,r\in(0,\,r_0):\;\gamma(r)\in(0,1)\quad\mathrm{and}\quad\limsup_{r\downarrow
			0}\,r^{-2d/(d-2)}\,\gamma(r)=\infty,
		\end{equation}
		the condition
		\begin{equation}\label{cndlebesgue}
		\lim_{|\by|\rightarrow\infty}\inf_{F\in\mathcal{M}_{\gamma(r)}(\by,r)}\int_{\mathcal{G}_r(\by)\setminus
			F}V(\e)\, \mathrm{d}\e=\infty
		\end{equation}
		is satisfied for any $r\in(0,r_0]$, then the spectrum of operator
		$H=-\Delta+V(\e)$ is discrete.
	\end{proposition}

\begin{remark}\label{remwithoutnonneg}
When proving the above statement, we did not use the non-negativity of the potential $V(\e)$ in order to show that the conditions of  Theorem \ref{thMazSh} follow from its conditions. 
\end{remark}

On the base of Proposition \ref{thcondlebesgue} we obtained the following claims:

\begin{proposition}\label{thcondlebesrear}\cite[ Proposition 4.7]{Zel2}
	If for some $r_0>0$ and a function $\gamma(r)$, satisfying conditions \eqref{cndtildgam1}, 
	the condition
	\begin{equation}\label{cndlebesrear}
	\lim_{|\by|\rightarrow\infty} V^\star(\delta(r),\by,r)=\infty
	\end{equation}
	is fulfilled  for any $r\in(0,r_0]$ with $\delta(r)=\gamma(r)\mathrm{mes}_d(\mathcal{G}_r(0)$, then the spectrum of operator
	$H=-\Delta+V(\e)$ is discrete.
\end{proposition}	
	
\begin{proposition}\label{thdiscLagr}\cite[ Theorem 5.2]{Zel2}
	Suppose that  $V(\e)\ge 0$,
	$V\in L_{2,loc}(\R^d)$ and the condition 
	\begin{equation}\label{cnddisclarg}
	\lim_{|\by|\rightarrow\infty}\Big(\mathrm{E}_{\by,r}(V)-\sqrt{\gamma(r)}\cdot\mathrm{Dev}_{\by,r}(V)\Big) =+\infty
	\end{equation}
	is satisfied for some $r_0>0$ and any $r\in(0,r_0)$, where $\gamma(r)$ satisfies conditions \eqref{cndtildgam1}. Then the spectrum of operator
	$H=-\Delta+V(\e)$ is discrete.
\end{proposition}

We will use also the following claim:

\begin{proposition}\cite[Proposition B.1]{Zel2}\label{lmest}
	Suppose that $d>2$ and  $\Omega$ is a bounded  domain in $\R^d$.
	
	\noindent	(i) If $W\in L_{d/2}(\Omega)$, then
	for any $u\in C_0^\infty(\Omega)$
	\begin{equation}\label{Sobineq}
	|\int_\Omega W(\e)|u(\e)|^2\, \mathrm{d}\e|\le C^2(d)\Vert
	W\Vert_{d/2}\int_\Omega|\nabla u(\e)|^2\, \mathrm{d}\e,
	\end{equation}
	where
	\begin{equation}\label{dfC}
	C(d)=\sqrt{\frac{1}{\pi
			d(d-2)}}\Big(\frac{\Gamma(d)}{\Gamma(d/2)}\Big)^{1/d};
	\end{equation}
	\vskip2mm
	
	\noindent	(ii)  Suppose that there is a vector field
	$\vec F(\e)$ on $\Omega$ belonging to $L_d(\Omega)$ and
	satisfying divergence equation \eqref{divereq} in the sense of distributions, where
	$W\in L_1(\Omega)$. Then for any $u\in C_0^\infty(\Omega)$
	\begin{equation}\label{estintlor1}
	|\int_{\Omega} W(\e)|u(\e)|^2\,\mathrm{d}\e|\le 2\,C(d)\,\Vert
	\vec F\Vert_d\int_{\Omega}|\nabla u(\e)|^2\, \mathrm{d}\e;
	\end{equation}
\end{proposition}
		
	\section{Main results} \label{sec:meinres}
	\setcounter{equation}{0}
	
	\subsection{Necessary conditions}\label{subsec.nesscond}
	
	Let us introduce some notation.  Consider  the
	spherical layer, centered at $\by\in\R^d$ $\;
	L_{s,t}(\by)=B_t(\by)\setminus\bar B_s(\by)\;(0< s<t)$ and the triangle
	in plane $\R^2:\,$ $\Delta_r=\{\langle s,t\rangle\in\R^2:\,0< s<t<r\}$.
Consider also the families of functions
	\begin{equation*}\label{dfcalW}
	{\mathcal W}_\by(t)=\int_{ B_{t}(\by)}V(\e)\,
	\mathrm{d}\e,
	\end{equation*}
	\begin{equation*}\label{dfcalV}
	{\mathcal V}_\by(s,t)=\mathcal{W_\by}(t)-\mathcal{W_\by}(s)=\int_{ L_{s,t}(\by)}V(\e)\,
	\mathrm{d}\e,
	\end{equation*}
	 
	The following necessary conditions for discreteness and boundedness below of the spectrum of $H$ are valid:
	\begin{theorem}\label{thnesscond1}
		Suppose that the potential $V(\e)$ is locally H\"older continuous
		and the spectrum of the operator $H$ is discrete and bounded below
		Then for any $r>0$ 
		\vskip2mm
	(i)	the family of functions $\mathcal{W_\by}(t)$
		tends to infinity in the measure $\mathrm{mes}_1$ on interval $(0,r)$ as
		$|\by|\rightarrow\infty$, i.e. 
		\begin{equation}\label{mesconvinfty1}
		\forall\,A>0\,:\quad\lim_{|\by|\rightarrow\infty}\mathrm{mes}_{\,1\,}\big(\{t\in(0,\,r):\;{\mathcal W}_\by(t)\le
		A\}\big)=0;
		\end{equation}
		\vskip2mm
		(ii) the family of functions $\mathcal{V_\by}(s,t)$
		tends to infinity in the measure $\mathrm{mes}_2$ on triangle $\Delta_r$ as
		$|\by|\rightarrow\infty$, i.e.
		\begin{equation}\label{mesconvinfty}
		\forall\,A>0\,:\quad\lim_{|\by|\rightarrow\infty}\mathrm{mes}_{\,2\,}\big(\{\langle s,t\rangle\in\Delta_r:\;{\mathcal V}_\by(s,t)\le
		A\}\big)=0
		\end{equation}
	\end{theorem}

\begin{remark}\label{remcompar}
In the case where $V(\e)\ge 0$ each function $\mathcal{W_\by}(t)$ is monotone non-decreasing. Using this fact, it is easy to show that conclusions (i) and (ii) of Theorem \ref{thnesscond1} are  equivalent to each other and to condition \eqref{intVtendinf} with $\mathcal{G}_r(\by)=B_r(\by)$. 
\end{remark}

\begin{remark}\label{remcompar1}
If it turned out that in the case, where $V(\e)$ is not supposed to be bounded below, one of conclusions of  Theorem \ref{thnesscond1} implies  another,  we could remove the latter  from  formulation of the theorem without weakening it. But Example \ref{ex1}  (Section \ref{sec:examples}) shows that   conclusion (ii) does not imply (i). On the other hand, so far we have not succeeded to fully substantiate an example which would show that (i) does not imply (ii) (Remark \ref{remtry}).
\end{remark}
	
\subsection{Sufficient conditions}\label{subsec.suffcond}

We need the following notion:
\begin{definition}\label{defrhocov}
A sequence of domains $\{\mathcal{G}_{r}(\by_k) \}_{k=1}^\infty$ forms $\rho$-covering  of $\R^d$ for some $\rho>0$, if

(a)
\begin{equation*}
|\by_k|\rightarrow\infty\quad \mathrm{for}\quad k\rightarrow\infty;
\end{equation*}
\vskip2mm
(b) for any $\by\in\R^d$ there is $j$ such that $B_\rho(\by)\subseteq\mathcal{G}_r(\by_j)$.
\end{definition}

In this section we will be based on following claim: 
\begin{proposition}\label{lmperturbdisc}
	Suppose that $\theta_1>0$, $\theta_2>0$, $\theta_1+\theta_2=1$ and in each domain $\mathcal{G}_{r_0}(\by)$
	the potential $V(\e)$ of the Schr\"odinger operator
	$H=-\Delta+V(\e)\cdot$ has the form $V(\e)=V_1^{(\by)}(\e)+V_2^{(\by)}(\e)$, where $V_k^{(\by)}\in L_{\infty}(\mathcal{G}_{r_0}(\by))\,(k\in\{1,2\})$. Consider the operators 
	$H_k^{(\by)}=-\Delta+\frac{1}{\theta_k}V_k^{(\by)}(\e)\cdot$ and the numbers 
	\begin{equation}\label{dflamky}
	\lambda_0^{(k)}(\by)=\inf_{u\in C_0^\infty(\mathcal{G}_{r_0}(\by),\,u\neq
		0}\frac{(H_k^{(\by)} u,u)_{\mathcal{G}_{r_0}(\by)}}{\Vert u\Vert_{\mathcal{G}_{r_0}(\by)}^2}
	\end{equation}
\vskip2mm	
	\noindent(i)  If 
	\begin{equation}\label{lamkyboundbel1}
	\liminf_{|\by|\rightarrow\infty}\lambda_0^{(k)}(\by)>-\infty\quad ,(k\in\{1,2\}), 
	\end{equation}
	 then the operator $H$ is bounded below;
	\vskip2mm
	\noindent(ii)  If 
	\begin{equation*}\label{lamkyboundbel}
	\lim_{|\by|\rightarrow\infty}\lambda_0^{(1)}(\by)=+\infty,
	\end{equation*}
	and \eqref{lamkyboundbel1} is fulfilled for $k=2$, then the spectrum of $H$ is discrete and bounded below.
\vskip2mm
\noindent(iii)  If a sequence of domains $\{\mathcal{G}_{r_0}(\by_k) \}_{k=1}^\infty$ forms $\rho$-covering  of $\R^d$, then claims (i)	 and (ii) are valid with $\by_k$ instead of $\by$.
	
\end{proposition}

	Let $V_+(\e)$ and $V_-(\e)$ be the positive and, respectively,
negative parts of the potential $V(\e)$, i.e.,
\begin{equation}\label{pospart}
V_+(\e)=\max\{V(\e),\,0\},\quad V_-(\e)=\max\{-V(\e),\,0\}.
\end{equation}
Denote
\begin{equation}\label{dfKd}
K(d)=\frac{1}{C^2(d)},
\end{equation}
where $C(d)$ is expressed by (\ref{dfC}).
The following claim is  valid:

\begin{theorem}\label{thgenermolch1}
	Suppose that for some $r_0>0$ 
	\begin{equation}\label{condbndblc}
	\limsup_{|\by|\rightarrow\infty}\Big(\int_{\mathcal{G}_{r_0}(\by)}(V_-(\e))^{d/2},
	\mathrm{d}\e\Big)^{2/d}<K(d),
	\end{equation}
	Then
		\vskip2mm
	\noindent$(i)$ operator $H$ is
	bounded below;;
	\vskip2mm
	
	\noindent$(ii)$ the spectrum of operator $H$ is discrete, if for $V_+(\e)$ conditions of one of four claims
	are fulfilled: either  Theorem \ref{thMazSh}, or  Proposition \ref{thcondlebesgue}, or  Proposition \ref{thcondlebesrear}, or  Proposition \ref{thdiscLagr}.
\end{theorem}

\begin{remark}\label{remBenFort}
	In \cite[Theorem 2.2 ]{Ben-Fort1} the semi-boundedness of $H$ is guarantied by the condition $V_-\in L_{d/2}(\R^d)$, which obviously implies  condition \eqref{condbndblc}.
\end{remark}

The following claims are based on Theorem \ref{thgenermolch1}:
\begin{corollary}\label{corNtendinfty}
	Suppose that in the formulation of Theorem \ref{thgenermolch1}
	instead of \eqref{condbndblc} the following condition is fulfilled : for any $\by\in\R^d$ $\;V_-\in L_{d/2}(\mathcal{G}_{r_0}(\by))$ and the integrals
	\begin{equation*}
	I_{\by}(N)=\int_{\mathcal{G}_{r_0}(\by) }(V(\e)+N)_-^{d/2}\,\mathrm{d}\e
	\end{equation*}	
	tend to zero as $N\rightarrow\infty$ uniformly with respect to $\by\in\R^d$. Then the spectrum of operator $H$ is discrete and bounded below.
 \end{corollary}

\begin{corollary}\label{thgenermolch}
	The condition of Corollary \ref{corNtendinfty} is fulfilled, if for some $r_0>0$ there is a nonnegative function $\psi\in
	L_{d/2}(\mathcal{G}_{r_0}(0))$  such that
\begin{equation}\label{cndbndbld2}
	\forall\,\by\in \R^d,\,\e\in \mathcal{G}_{r_0}(\by)\,:\quad
	V(\e)\ge-\psi(\e-\by).
	\end{equation}
\end{corollary}

For each $\by\in\R^d$ consider the function 
\begin{equation}\label{defWy}
W^{(\by)}(\e)=V(\e)-\mathrm{E_{\by,r_0}}(V),
\end{equation}
 defined on 
$\mathcal{G}_{r_0}(\by)$. The following claim takes into account interplay between positive and negative parts of the potential $V(\e)$:

\begin{theorem}\label{theorpert}
	Suppose that for some $r_0>0$ 
	\vskip2mm
	
	(a)  a sequence of domains $\{\mathcal{G}_{r_0}(\by_k) \}_{k=1}^\infty$ forms $\rho$-covering  of $\R^d$;
	\vskip2mm
	
	(b)
	
	\begin{equation}\label{cndpert}
	\bar D=\limsup_{k\rightarrow\infty}D_d(W^{(\by_k)},\by_k,r_0)<
	1/(2C(d)),
	\end{equation}
Then
\vskip2mm
\noindent$(i)$ If
\begin{equation}\label{cndsemboundEy}
\liminf_{k\rightarrow\infty}\mathrm{E_{\by_k,r_0}}(V)>-\infty,
\end{equation}
operator $H$ is bounded below;;
\vskip2mm
\noindent(ii)  If 
\begin{equation}\label{cndEyinfty}
\lim_{k\rightarrow\infty}\mathrm{E_{\by_k,r_0}}(V)=+\infty.
\end{equation}
the spectrum of $H$ is discrete and bounded below.
\end{theorem}

Recall that the quantity $D_s(W,\,\by,\,r)$	 is defined in the line below definition \eqref{defDsW}. 

 The following claim is based on Theorem \ref{theorpert}  with the choice of vector fields $\vec F(\e)$ as minimizers in \eqref{defDsW} with $\Omega=\mathcal{G}_{r_0}(\by)$ under assumption that $\vec F(\e) $ is orthogonal to normal $\vec n(\e)$ at each point $\e\in\partial\mathcal{G}_{r_0}(\by)$. 
 For this choice we use Proposition \ref{prtranspr} and Remark \ref{remtrasport} connected with a divergence constrained transportation problem. On each domain $\mathcal{G}_{r_0}(\by)$ consider the following nonhomogeneous $d^\prime$-Laplace equation 
\begin{equation}\label{dLapleq}
-\mathrm{div}\big(|\nabla u|^{d^\prime-2}\nabla u\big)=d^{1/(d-1)}W^{(\by)}(\e)
\end{equation}
with $W^{(\by)}(\e)$, defined by \eqref{defWy} and $d^\prime=d/(d-1)$.  Also consider the Neumann boundary
condition
\begin{equation}\label{neumboundcnd1}
|\nabla u|^{d^\prime-2}\frac{\partial u}{\partial\vec n}\vert_{\partial \mathcal{G}_{r_0}(\by)}=0,
\end{equation}
We will suppose that the boundary $\partial\mathcal{G}$ of domain $\mathcal{G}$ is smooth. Then in view of definition  \eqref{dfcalGry} the domains $\mathcal{G}_{r_0}(\by)$ have the same property.

The promised claim is:

\begin{corollary}\label{thchicpert}
	Suppose that condition (a)  of Theorem \ref{theorpert} is fulfilled. Then
	claims (i) and (ii) of this theorem are valid, if instead of condition \eqref{cndpert} the following one 
\begin{equation}\label{cndpert2}
\limsup_{k\rightarrow\infty}\Big(\frac{1}{d^{1/(d-1)}}\int
_{\mathcal{G}_{r_0}(\by_k)}
|\nabla u_k(\e)|^{d^\prime}\,\mathrm{d}\e\Big)^{1/d}<\frac{1}{2C(d)}.
\end{equation}
is fulfilled, where  $u_k(\e)$ is a solution of the Neumann problem 
\eqref{dLapleq}-\eqref{neumboundcnd1} with $\by=\by_k$.
\end{corollary}

\begin{remark}\label{remcndsembound}
In Section \ref{sec:examples} we have constructed an example (Example \ref{ex3}) of the potential $V(\e)$ satisfying conditions of  Corollary \ref{thchicpert}, but condition \eqref{condbndblc} of Theorem \ref{thgenermolch1} is not fulfilled for it
\end{remark}

\section{Proof of main results} \label{sec:proofmainres}
\setcounter{equation}{0}

\subsection{Localization principle}\label{subsec:localizprinc}

Let $\Omega$ be an open domain in $\R^d$ whose closure is compact.
 Consider the quantity:
\begin{eqnarray}\label{lowboundspect}
&&\lambda_0(\Omega):=\inf_{u\in C_0^\infty(\Omega),\,u\neq
	0}\frac{(H u,u)_\Omega}{\Vert u\Vert_\Omega^2}=\nonumber\\
&&\inf_{u\in C_0^\infty(\Omega),\,u\neq 0}\frac{\int_\Omega(|\nabla
	u(\e)|^2+V(\e)|u(\e)|^2)\, \mathrm{d}\e}{\int_\Omega|u(\e)|^2\,
	\mathrm{d}\e}.
\end{eqnarray}
In essence $\lambda_0(\Omega)$ is the minimal eigenvalue of the
generalized Dirichlet  boundary problem in the domain $\Omega$:
$Hu=\lambda u\;u\vert_{\partial\Omega}=0$. As is known, the spectrum
of this problem id discrete.  we denote briefly
$\lambda_0(\by,r)=\lambda_0(B_r(\by))$. The following localization
principle was established in \cite[Theorem 1.1]{K-Sh} for the case
of  magnetic Schr\"odinger operator, but in particular it is true
also in absence of the magnetic field.
\begin{proposition}\label{prloc} 
	(i) The operator $H$ is  bounded below if and only if for some $r>0$
	\begin{equation*}
	\liminf_{|\by|\rightarrow\infty}\lambda_0(\by,r)>-\infty;
	\end{equation*}
	
		(ii) The operator $H$ has discrete spectrum if and only if for some $r>0$
	\begin{equation*}
	\lim_{|\by|\rightarrow\infty}\lambda_0(\by,r)=+\infty.
	\end{equation*}
\end{proposition}

Notice that firstly the localization principle for the one-dimensional case ($d=1$) was obtained in \cite{Is}.

\subsection{Localization principle via Riccati
PDE}\label{subsec:locRicc}

Let's agree to say that a function $u(\e)$ is a classical solution of some second order PDE in the closure $\bar\Omega$ of some open domain $\Omega\subseteq\R^d$, if it belongs to $C^2(\Omega)\cap C(\bar \Omega)$ and satisfies this equation in $\Omega$.

. Consider the stationary Schr\"odinger equation 
\begin{equation}\label{Schreq}
-\Delta u+V(\e)u=\lambda u,
\end{equation}
in a bounded closed domain $\bar\Omega\in\R^d$.
The change of variable $v=-\ln u$ leads to 
Riccati PDE
\begin{equation}\label{Ricceq}
-\Delta v+|\nabla v|^2+\lambda=V(\e),
\end{equation}
 We see that equation \eqref{Schreq} has a classical
positive solution in  $\bar \Omega$ if and only
if equation~\eqref{Ricceq} has there a  classical solution.  

The following modification of the localization principle is valid:

\begin{proposition}\label{prlocRicc}
	Assume that the potential $V(\e)$ is locally H\"older continuous.
Then the spectrum of $H$ is discrete and bounded below
	if and only if for some $r>0$ and $R>0$ there exists a real-valued
	function $\lambda(\by)\,(\by\in\R^d\setminus B_R(0))$ such that
	\begin{equation}\label{lambdytendsinf}
	\lim_{|\by|\rightarrow\infty}\lambda(\by)=+\infty
	\end{equation}
	 and  Riccati PDE 
\begin{equation}\label{Ricceq1}
-\Delta v+|\nabla v|^2+\lambda(\by)=V(\e),
\end{equation}	
	 has a classical solution in each ball
	$\bar B_r(\by)$ with $|\by|\ge R$.
\end{proposition}
\begin{proof}
In view of Proposition \ref{prloc} and Lemma \ref{lmpositsol}, the spectrum of $H$ is discrete and bounded below if and only if there exists a real valued function $\lambda(\by)$ such that $\lambda(\by)<\lambda_0(\by,r)$, condition \eqref{lambdytendsinf} is fulfilled and equation $-\Delta u+V(\e)u=\lambda(\by) u$  
has in $\bar B_r(\by)$ a classical positive solution $u(\e)$. Then the function $v=-\ln u$ 
 is a classical solution of \eqref{Ricceq1} in $\bar B_r(\by)$.
\end{proof} 

In the proof of Proposition \ref{prlocRicc} we have used the following claim:
\begin{lemma}\label{lmpositsol}
Let $\Omega\subset\R^d$ be a bounded domain with a smooth boundary and $\lambda_0(\Omega)$ is defined by \eqref{lowboundspect}. Assume that $V(\e)$ is H\"older continuous in $\bar\Omega$. Then equation \eqref{Schreq} has in $\bar\Omega$ a classical positive solution if and only if 
\begin{equation}\label{lamblesslamb0Om}
\lambda<\lambda_0(\Omega).
\end{equation}
\end{lemma}
\begin{proof}
Assume that \eqref{lamblesslamb0Om} is valid. Hence
\begin{equation*}
\inf_{u\in C_0^\infty(\Omega),\,u\neq
	0}\frac{(H u-\lambda u,\,u)_\Omega}{\Vert u\Vert_\Omega^2}=\lambda_0(\Omega)-\lambda>0.
\end{equation*}
Therefore in particular for every domain $D$ with $\bar D\subset\Omega$ the only solution of \eqref{Schreq} in $D$, belonging to $C^2(D)\cap C^0(\bar D)$ and vanishing on $\partial D$, is $u\equiv 0$.
 Then by \cite[Lemma 2.3]{Moss-Piep},
  equation \eqref{Schreq} has in $\bar\Omega$ a classical positive solution $u(\e)$. Conversely, suppose that equation \eqref{Schreq} has such a solution. Assume on the contrary that $\lambda\ge\lambda_0(\Omega)$. Since $V(e)$ is locally H\"older continuous, then 
 by the definition of 
 $\lambda_0(\Omega)$ and \cite[Theorem 6.15]{Gil-Tr}  equation $-\Delta e+V(\e)e=\lambda_0(\Omega)e$ has in $\bar\Omega$ a nontrivial solution $e(\e)$, belonging to $C^{2,\alpha}(\bar\Omega)$ for some $\alpha\in(0,1)$ and satisfying the boundary condition $e(\e)\vert_{\e\in\partial\Omega}=0$. Applying the comparison theorem \cite[Theorem 1]{Cl-Sw} to the last equation and to \eqref{Schreq} , we obtain that $u(\e)$ has at least one zero in $\bar\Omega$. The latter contradicts the supposed property of $u(\e)$.
\end{proof}

\subsection{Proof of Theorem \ref{thnesscond1}}
\begin{proof}
Assume that the spectrum of $H$ is discrete and bounded below.
Then by Proposition~\ref{prlocRicc}, for any $r>0$ there exist a
real-valued function $\lambda(\by)\,(\by\in\R^d)$ such that
\begin{equation}\label{lambdyconvinft}
\lim_{|\by|\rightarrow\infty}\lambda(\by)=+\infty
\end{equation}
and for any $\by\in\R^d$ there is a function $v^{(\by)}\in C^2( \bar B_r(\by))$
satisfying in $B_r(\by)$ Riccati PDE, i.e.,
\begin{equation}\label{satRicceq}
-\Delta v^{(\by)}(\e)+|\nabla v^{(\by)}(\e)|^2+\lambda(\by)=V(\e)
\end{equation}
	 Denote $w^{(\by)}(t)=\int_{\partial B_t(\by)}
\frac{\partial v^{(\by)}(\e)}{\partial\vec n}\,ds(\e)$. As is clear,
$w\in C[0,r]$. Using Green's formula, we have for $0\le s< t\le r$:
\begin{equation}\label{Greenform}
\int_{ L_{s,t}(\by)}\Delta v^{(\by)}(\e)\,
\mathrm{d}\e=\int_{\partial L_{s,t}(\by)} \frac{\partial
	v^{(\by)}(\e)}{\partial\vec n}\,ds(\e)=w^{(\by)}(t)-w^{(\by)}(s)
\end{equation}
On the other hand, using Schwartz's inequality, we have
\begin{eqnarray}\label{usSchvineq}
&&\int_s^t(w(^{(\by)}(\xi))^2\,d\xi=\int_s^t\Big(\int_{\partial
		B_\xi(\by)}\nabla v^{(\by)}(\e)\cdot\vec
	n(\e)\,ds(\e)\Big)^2\,d\xi\le\nonumber\\
	&&\sigma_{d}r^{d-1}\int_{ L_{s,t}(\by)}|\nabla
	v^{(\by)}(\e)|^2\, \mathrm{d}\e,
	\end{eqnarray}

Let us prove claim (i). It is clear that in order to prove \eqref{mesconvinfty1},
it is enough to show that for any $\delta\in(0,r)$ and $A>0$
\begin{equation*}
\lim_{|\by|\rightarrow\infty}\mathrm{mes}_d(F_\by^\delta(A))=0,
\end{equation*}
where 
\begin{equation*}
F_\by^\delta(A)=\{t\in[\delta,\,r]:\;{\mathcal W}_\by(t)\le
A\cdot 
\mathrm{mes}_d(B_t(0))\}
\end{equation*}
Let us put $s=0$ in \eqref{usSchvineq} and denote 
\begin{equation}\label{defut}
u^{(\by)}(t)=(\omega_dr^{d-1})^{-1}w^{(\by)}(t).
\end{equation}
 Then in view of \eqref{satRicceq}-\eqref{usSchvineq}, we obtain:
\begin{eqnarray}
	&&\mathcal{W_\by}(t)=\int_{ B_{t}(\by)}V(\e)\,
	\mathrm{d}\e\ge
	\sigma_{d}r^{d-1}\Big(-u^{(\by)}(t)+\int_0^t(u^{(\by)}(\xi))^2\,d\xi\Big)+\nonumber\\
	&&\lambda(\by)\mathrm{mes}_d(B_t(0))\nonumber.
\end{eqnarray}
Notice that in view of \eqref{lambdyconvinft} there is $R>0$ such that $\lambda(y)>A$ for $\by\in\R^d\setminus B_R(0)$. Hence, since for $t\in[\delta,r]$ $\mathrm{mes}_d(B_t(0))\ge t\cdot\omega_d\delta^{d-1}/d $, the set $F_\by^\delta(A)$ is contained in the set
\begin{eqnarray}
\tilde F_\by^\delta(A)=\{ t\in[\delta ,r]:\;u^{(\by)}(t)\ge
	\int_0^t(u^{(\by)}(\xi))^2\,
	\mathrm{d}\xi+\tilde\lambda_A^\delta(\by)\cdot t\}\nonumber,
	\end{eqnarray}
where 
\begin{equation*}
\tilde\lambda_A^\delta(\by)=\frac{(\lambda(\by)-A)\delta^{d-1}}{d\cdot r^{d-1}.}
\end{equation*}
Then in view of\eqref{lambdyconvinft}, Lemma \ref{lmintineq} implies that
\begin{equation}
\lim_{|\by|\rightarrow\infty}\mathrm{mes}_d(\tilde F_\by^\delta(A))=0.
\end{equation}
Therefore the desired equality \eqref{mesconvinfty1} is valid.	

 Now let us prove claim (ii). It is clear that in order to prove \eqref{mesconvinfty},
	it is enough to show that for any $\delta\in
	(0,r]$ and $A>0$ 
	\begin{equation}\label{limmesEdelta}
	\lim_{|\by|\rightarrow\infty}\mathrm{mes}_d(E_\by^\delta(A))=0,
	\end{equation}
	where
	\begin{eqnarray}
		&&E_\by^\delta(A)=\{\langle s,t\rangle\in\Delta_r^\delta:\;\mathcal{V_\by}(s,t)\le
		A\,\mathrm{mes}_{\,d}( L_{s,t}(0))\},\nonumber\\
		&&\Delta_r^\delta=\{\langle s,t\rangle\in\R^2:\;\delta<s<t<r\}.\nonumber
	\end{eqnarray}
	Notice that for $\langle t,s\rangle\in \Delta_r^\delta$ $\;\mathrm{mes}_{\,d}( L_{s,t}(0))\ge\omega_d\delta^{d-1}(t-s)$.
Then in the similar manner as above we obtain from \eqref{satRicceq}-\eqref{usSchvineq} that for $\by\in\R^d\setminus B_R(0)$ the set $E_\by^\delta(A)$ is contained in the set
	\begin{equation*}
	\tilde
		E_\by^\delta(A)=\{\langle s,t\rangle\in\Delta_r^\delta:\;u^{(\by)}(t)-u^{(\by)}(s)\ge
		\int_s^t(u^{(\by)}(\xi))^2\,
		\mathrm{d}\xi+\bar\lambda_A^\delta(\by)(t-s)\}
		\end{equation*}
	where $u^{(\by)}(t)$  is defined by \eqref{defut} and 
	\begin{equation*}
	\bar\lambda_A^\delta(\by)=\frac{(\lambda(\by)-A)\delta^{d-1}}{r^{d-1}.}
	\end{equation*}
	Then Proposition~\ref{mainprAp2} implies that
	$\lim_{|\by|\rightarrow\infty}\mathrm {mes}_{\,2}(\tilde
	E_\by^\delta(A))=0$. Hence \eqref{mesconvinfty} is valid.
\end{proof}

\subsection{Proof of Proposition \ref{lmperturbdisc}}

\begin{proof}
	Let us represent:
	\begin{eqnarray}\label{reprH}
	&&\hskip-8mmH=-\Delta+V_1^{(\by)}(\e)\cdot+V_2^{(\by)}(\e)\cdot=\nonumber\\
&&\hskip-8mm	(-\theta_1\Delta+V_1^{(\by)}(\e)\cdot)+(-\theta_2\Delta+V_2^{(\by)}(\e)\cdot)=
	\theta_1H_1^{(\by)}+\theta_2H_2^{(\by)}.
	\end{eqnarray}
Then in view of \eqref{dflamky}
	\begin{eqnarray}
		&&\lambda_0(\by,r_0):=\inf_{u\in C_0^\infty(\mathcal{G}_{r_0}(\by)),\,u\neq
			0}
		\frac{(H u,u)_{\mathcal{G}_{r_0}(\by)}}{\Vert
			u\Vert_{\mathcal{G}_{r_0}(\by)}^2}\ge\theta_1\lambda_0^{(1)}(\by)+\theta_2\lambda_0^{(2)}(\by).\nonumber
	\end{eqnarray}
This inequality and Proposition \ref{prloc} imply claims (i)  and (ii) .	
Let us prove claim (iii).  By Definition \ref{defrhocov} of $\rho$-covering, for some $\rho>0$ and any $\by\in\R^d$ there is $j$ such that $B_\rho(\by)\subseteq\mathcal{G}_{r_0}(\by_j)$.
Then the obvious inequality
\begin{equation*}
\inf_{u\in C_0^\infty(B_\rho(\by)),\,u\neq
	0}\frac{(H u,u)_{B_\rho(\by)}}{\Vert
	u\Vert_{B_\rho(\by)}^2}\ge\lambda_0(\by_j,r_0)
\end{equation*}
and Proposition \ref{prloc} imply claim (iii).
\end{proof}		

\subsection{Proof of Theorem \ref{thgenermolch1}}

\begin{proof}
	(i) We have: $V(\e)=V_+(\e)-V_-(\e)$. Let us take $\theta_1>0$,
	$\theta_2>0$ such that $\theta_1+\theta_2=1$ and represent
	$H=\theta_1H_++\theta_2H_-$,
	where $H_+=-\Delta+(\theta_1)^{-1}V_+(\e)\cdot$ and
	$H_-=-\Delta-(\theta_2)^{-1}V_-(\e)\cdot$. It is clear that the
	operator $H_+$ us nonnegative. Using condition \eqref{condbndblc}
	we can choose $\theta_2\in(0,1)$ such that
	\begin{equation}\label{cndbndbld}
	\bar
	L=\limsup_{|\by|\rightarrow\infty}\Big(\int_{\mathcal{G}_{r_0}(\by)}(V_-(\e))^{d/2}\,
	\mathrm{d}\e\Big)^{2/d}<\theta_2\,K(d).
	\end{equation}
 By claim (i) of Proposition \ref{lmperturbdisc}, with 
 \begin{equation*}
  V_1^{(\by)}=V_+,\quad  V_2^{(\by)}=-V_-,\quad H_1^{(\by)}=H_+,\quad H_2^{(\by)}=H_-, 
 \end{equation*}
  in order to prove that the operator $H$ is bounded below it is enough
	to show that the operator $H_-$ is bounded below. By claim (i) of Proposition \ref{lmest}, 
	\begin{equation*}
	\int_{\mathcal{G}_{r_0}(\by)} V_-(\e)|u(\e)|^2\,
	\mathrm{d}\e\le C^2(d)\Vert
	V_-\Vert_{d/2,\,\mathcal{G}_{r_0}(\by)}\int_{\mathcal{G}_{r_0}(\by)}|\nabla u(\e)|^2\,
	\mathrm{d}\e 
	\end{equation*}
	for any $u\in C_0^\infty(\mathcal{G}_{r_0}(\by))$. Then, in view of
	\eqref{cndbndbld} and \eqref{dfKd} there is $R>0$ such that for
	any $\by\notin B_R(0)$ and $r\in(0,r_0]$ 
	\begin{eqnarray}
		&&(H_-u,\,u))_{\mathcal{G}_r(\by)}=\int_{\mathcal{G}_r(\by)}\big(|\nabla
		u(\e)|^2-(\theta_2)^{-1}V_-(\e)|u(\e)|^2\big)\, \mathrm{d}\e\ge\nonumber\\
		&&\big(1-\bar L\,(\theta_2\,K(d))^{-1}\big)\int_{\mathcal{G}_r(\by)}|\nabla
		u(\e)|^2\, \mathrm{d}\e\ge 0.\nonumber
	\end{eqnarray}
	This circumstance and claim (i) of Proposition \ref{prloc} imply
	that the operator $H_-$ is bounded below. Claim (i) is proven.
	 
	(ii)  By claim (ii) of Proposition \ref{lmperturbdisc} and claim (ii)  of Proposition \ref{prloc}, it is enough to show that the spectrum of operator $H_+$ is discrete. But this fact is guarantied by each of claims, mentioned in the formulation of claim (ii) of the present theorem. Notice that while  we use Propositions \ref{thcondlebesrear} and \ref{thdiscLagr}, we should apply the obvious properties:
	\begin{eqnarray}
	&&(\alpha\cdot V_+)^\star(t,\by,r)=\alpha\cdot V_+^\star(t,\by,r),\quad \mathrm{E}_{\by,r}(\alpha\cdot V_+)=\alpha\cdot \mathrm{E}_{\by,r}(V_+),\nonumber\\
	&& \mathrm{Dev}_{\by,r}(\alpha\cdot V_+)=\alpha\cdot\mathrm{Dev}_{\by,r}(V_+),
	 \end{eqnarray}
where $\alpha>0$.	
\end{proof}

\subsection{Proof of Corollary\ref{corNtendinfty}}
\begin{proof}
	In view of the condition, imposed on integrals $I_{\by}(N)$, there is $N>0$ such that for any
	$\by\in\R^d$ 
	\begin{equation*}
	\int_{\mathcal{G}_{r_0}(\by)}(V(\e)+N)_-^{d/2}\,
	\mathrm{d}\e\le(K(d))^{d/2}, 
	\end{equation*}
	where $K(d)$ is defined by \eqref{dfKd}.
	Hence the operator $H+NI$ satisfies the condition of claim (i) of
	Theorem \ref{thgenermolch1} and if condition \eqref{cndmolch1}
	holds, it satisfies the condition of claim (ii) of this theorem.
	Therefore the conclusions of these claims are valid.
\end{proof}

\subsection{Proof of Corollary \ref{thgenermolch}}
\begin{proof}
	Let us take $N>0$. Condition \eqref{cndbndbld2} and definitions \eqref{pospart} imply that
	for any $\by\in\R^d$ and $\e\in B_{r_0}(\by)\;$ $(V(\e)+N)_-\le
	(\psi(\e-\by)-N)_+$. Hence
	\begin{equation*}
	\int_{\mathcal{G}_{r_0}(\by)}(V(\e)+N)_-^{d/2}\,
	\mathrm{d}\e\le\int_{\mathcal{G}_{r_0}(0)}(\psi(\e)-N)_+^{d/2}\, \mathrm{d}\e.
	\end{equation*}
	Since $\psi\in L_{d/2}(\mathcal{G}_{r_0}(0))$, then
	$\lim_{N\rightarrow\infty}\int_{\mathcal{G}_{r_0}(0)}(\psi(\e)-N)_+^{d/2}\,
	\mathrm{d}\e=0$. Hence the condition of Corollary \ref{corNtendinfty} is fulfilled.
\end{proof}

\subsection{Proof of Theorem \ref{theorpert}}
\begin{proof}
	In order to use Proposition \ref{lmperturbdisc}, let us take
	\begin{eqnarray}
&&	H_1^{(\by_k)}=-\Delta+\theta_1^{-1}\mathrm{E_{\by_k,r_0}}(V)\cdot,\quad H_2^{(\by_k)}=-\Delta+
	\theta_2^{-1}W^{(\by_k)}(\e)\cdot,\nonumber\\
&&	\theta_1>0,\quad \theta_2>0,\quad \theta_1+\theta_2=1.
	\end{eqnarray}
	Assume that $\theta_2\in[2C(d)\bar D,\,1)$.
	In view of definition \eqref{defDsW} and
	condition  \eqref{cndpert},
	there exists $R>0$ such that for any $k$
	there is a vector field
	$\vec F_{\by_k}\in \mathcal{D}_d(W^{(\by_k)},\by_k)$, for which
	$\Vert \vec F_{\by_k}\Vert_d\le\theta_2/(2C(d))$. Taking into
	account that
	\begin{equation*}
	W^{(\by_k)}(\e)=\mathrm{div}\,\vec F_{\by_k}(\e)\;(\e\in \mathcal{G}_{r_0}(\by_k))
	\end{equation*}
	and using claim (ii) of Proposition \ref{lmest}, we obtain for $u\in C_0^\infty(\mathcal{G}_{r_0}(\by_k))$:
	\begin{eqnarray}
		|\int_{\mathcal{G}_{{r_0)}}(\by_k)}W^{(\by_k)}(\e)|u(\e)|^2\, \mathrm{d}\e|\le
		\theta_2\int_{\mathcal{G}_{r_0}(\by_k)}|\nabla u(\e)|^2\,\mathrm{d}\e,\nonumber
	\end{eqnarray}
	hence
	\begin{eqnarray}
&&(H_2^{(\by_k)} u,u)_{\mathcal{G}_{r_0}(\by_k)}=\nonumber\\
&&\int_{\mathcal{G}_{r_0}(\by_k)}\Big(|\nabla
u(\e)|^2+\theta_2^{-1}W^{(\by_k)}(\e)|u(\e)|^2\Big)\,\mathrm{d}\e	\ge 0
	\end{eqnarray}
	Furthermore,
	\begin{eqnarray}
	&&(H_1^{(\by_k)} u,u)_{\mathcal{G}_{r_0}(\by_k)}=
	\nonumber\\
	&&\int_{\mathcal{G}_{r_0}(\by)}\nabla u(\e)|^2\,\mathrm{d}\e+\theta_1^{-1}\mathrm{E_{\by_k,r_0}}|(V)
	\int_{\mathcal{G}_{r_0}(\by)}|u(\e)|^2\,\mathrm{d}\e\ge
	\nonumber\\
	&&\theta_1^{-1}\mathrm{E_{\by_k,r_0}}|(V)
	\int_{\mathcal{G}_{r_0}(\by_k)}|u(\e)|^2\,\mathrm{d}\e\nonumber
	\end{eqnarray}
	These estimates, condition (a), conditions \eqref{cndsemboundEy}, \eqref{cndEyinfty} with $\by=\by_k$ and claim (iii) of Proposition \ref{lmperturbdisc}
imply the desired claims.
\end{proof}

\subsection{Proof of Corollary \ref{thchicpert}}
\begin{proof}
In view of definition \eqref{defWy}, 
\begin{equation*}
\int_{\mathcal{G}_{r_0}(\by_k)} W^{(\by_k)}(\e)\,\mathrm{d}\e=0.
\end{equation*}
Then using Proposition \ref{prtranspr} for the optimal choice of the vector field $\vec F(\e)$ on each domain $\mathcal{G}_{r_0}(\by_k)$, 
 we obtain that conditions of  Theorem \ref{theorpert}  are fulfilled. Hence the conclusions of this theorem are valid. 
\end{proof}

\section{Examples}\label{sec:examples}
\setcounter{equation}{0}

\begin{example}\label{ex1}
	The following example was promised in Remark \ref{remcompar1}.
	  We will show that conclusion (ii) of Theorem \ref{thnesscond1} does not imply  (i).  To do this, we  should construct a potential $V(\e)$ that satisfies (ii) but not (i). 
	   First construct it in the ball $B_{r_0}(0)$. Consider the following sequence of functions, defined in $[0, r_0]$:
	\begin{eqnarray}\label{defSn}
S_n(x)=\left\{\begin{array}{ll}
-\frac{n^2}{r_0}x&\mathrm{for}\quad x\in[0,\,r_0/n)\\
-\frac{n}{r_0(1-r_0/n)}(r_0-x)&\mathrm{for}\quad x\in[r_0/n,\,r_0]
\end{array}
\right.\nonumber
	\end{eqnarray}
	Since $S_n(x)\le 0$ in $[0,\,r_0]$, it cannot tend to infinity in measure $\mathrm{mes}_1$. On the other hand, it is easy to see that the sequence 
\begin{equation*}
X_n(t, s)=S_n(t)-S_n(s)=\frac{n}{r_0(1-r_0/n)}(t-s)\quad(r_0/n<s<t) 
\end{equation*}	
	 tends to infinity in measure $\mathrm{mes}_2$ on the triangle $\Delta_{r_0}=\{\langle s,t\rangle\in\R^2:\,0< s<t<r_0\}$.
	 Now we can construct the desired potential $V(\e)$. Consider the function $\tilde V_n(x)$ such that
	 \begin{equation*}\label{eqtiln}
	 \omega_d\int_0^x\tilde V_n(\rho)\rho^{d-1}\,\mathrm{d}\rho =S_n(x),
	 \end{equation*}
	 hence
	 \begin{equation*}\label{dftildVn}
	 \tilde V_n(x)=\frac{S_n^\prime(x)}{\omega_d x^{d-1}}.
	 \end{equation*}
 For each $\vec\ell\in\Z^d$ consider the ball $B_{r_0}(\vec\ell)$ with $r_0<1/2$. It is clear that these balls are disjoint. Also consider the disjoint cubes 
	 \begin{equation*}\label{disjcub}
	 Q_{1/2}(\vec\ell)=Q_{1/2}(0)+\{\vec\ell\}, \quad\mathrm{where}\quad Q_{1/2}(0)=\prod_{k=1}^d. [-1/2,\,1/2\big)\quad 
	 \end{equation*}
	 They cover $\R^d$ and each of them contains the unique ball $B_{r_0}(\vec\ell)$. Let us define the potential  $V(\e)$ in each cube $Q_{1/2}(\vec\ell)\;(\vec\ell\in\Z^d)$ in the following manner:
	 \begin{eqnarray}
	 V(\e)=\left\{\begin{array}{ll}
	 \tilde V_n(|\e-\vec\ell|)&\mathrm{for}\quad \e\in B_{r_0}(\vec\ell),\\
	 \tilde V_n(r_0)&\mathrm{for}\quad \e\notin B_{r_0}(\vec\ell),
	 \end{array}
	 \right.\nonumber
	 \end{eqnarray}
	 where $n=|\vec\ell|_\infty$. In view of the above considerations, conclusion (i) of Theorem \ref{thnesscond1} is not valid for $V(\e)$ \footnote{In particular,  by Theorem \ref{thnesscond1} this means that either the operatot $H$ is not bounded below, or its spectrum is not discrete} and  conclusion (ii) is valid on the balls $B_{r_0}(\vec\ell)$.  It is easy to see  that the latter is valid for any sequence of balls that go to infinity and  have any fixed radius.
	\end{example}
\begin{remark}
We see that the potential constructed above is not continuous at points belonging to some hypersurfaces. But according to the formulation of Theorem \ref{thnesscond1} it should be locally. H\"lder continuous.  We could achieve this by smoothing the potential in vicinities of the discontinuity hypersurfaces, but we didn't want to create unnecessary technical difficulties.
\end{remark}
	  \begin{remark}\label{remtry}
	  	We tried to construct a potential $V(\e)$ that satisfies conclusion (i) of Theorem \ref{thnesscond1} but not  (ii)  on the base of following sequence of functions, defined in $[0,\,r_0]$.
	  	\begin{equation*}
	  	S_n(x)=A_nD_n(x), 
	  	\end{equation*},
	  	where
	  	\begin{equation*}\label{dfDnx}
	  	D_n(x)=U(nx/r_0),
	  	\end{equation*}
	  	\begin{equation*}\label{dfSnx}
	  	U(x)=\sum_{k=1}^{\infty} T(x-k+1/3),
	  	\end{equation*}
	  	\begin{eqnarray}\label{dfTx}
	  	T(\e)=\left\{\begin{array}{ll}
	  	3x&\mathrm{for}\quad x\in[0,\,1/3),\\
	  	1&\mathrm{for}\quad x\in[1/3,\,2/3),\\
	  	3(1-x)&\mathrm{for}\quad x\in[2/3,\,1],\\
	  	0&\mathrm{for}\quad x\ge 1\quad\mathrm{and}\quad x<0..
	  	\end{array}
	  	\right.\nonumber
	  	\end{eqnarray}
	  	and 
	  		\begin{equation*}\label{Antendsinfty}
	  	A_n\rightarrow+\infty\quad\mathrm{as}\quad n\rightarrow\infty.
	  	 \end{equation*}
	  Next we constructed the potential $V(\e)$ in the same manner as in the above example.
	 We showed that $V(\e)$  does not satisfy conclusion (ii) of Theorem \ref{thnesscond1} and satisfies (i) on the balls $B_{r_0}(\vec\ell)$. But so far we have not been able to find out whether the latter is true on any sequence of balls 
	 that go to infinity and have any fixed radius. 
	 \end{remark}

\begin{example}\label{ex2}
Here we construct the example, promised in Remark \ref{remsembound}.
		It is based on the example constructed in
	\cite{Zel}  which shows that Molchanov's condition and
	semi-boundedness of the one-dimensional Schr\"odinger operator
	$L=-\frac{d^2}{dx^2}+V(x)\cdot$ do not imply the discreteness of
	its spectrum. In the $d$-dimensional case with $d\ge 3$ we take the
	Schr\"odinger operator with the divided variables
	$H=-\Delta+W(\e)\cdot$, where 
	\begin{equation*}
	W(\e)=V(x_1)+\sum_{j=2}^dU(x_j),
	\end{equation*}
	where $U\in L_{\infty, loc}(\R)$, $U(x)>0$, 
	\begin{equation}\label{Uxinfty}
	U(x)\rightarrow\infty\quad \mathrm{as} \quad |x|\rightarrow\infty
	\end{equation}
	and $V(x)$ is the potential constructed in \cite{Zel}. Also
	consider the Schr\"odinger operator
	$M=-\frac{d^2}{dx^2}+U(x)\cdot$, which is positive and has the
	discrete spectrum.  Let us recall the construction of the
	potential $V(x)$. Consider the monotone sequence of positive
	numbers $(\alpha_n)_{n=1}^\infty$, such that
	\begin{equation}\label{cndalfn}
	\alpha_n\rightarrow\infty,\quad \frac{\alpha_n}{\sqrt{\alpha_n}+\sqrt{\alpha_{n+1}}}\rightarrow\infty\quad\mathrm{as}\quad n\rightarrow\infty,
	\end{equation}
	and the
	sequence of semi-intervals covering the axis $\R$
	\begin{eqnarray}
		&&\hskip-5mm\dots
		(m_{-1},\,l_{-3}];\,(l_{-3},\,l_{-2}];\,(l_{-2},\,m_0];\,(m_0,\,l_{-1}];\,(l_{-1},\,l_{0}];\,(l_{0},\,m_1];\,(m_1,\,l_1];\nonumber\\
		&&\hskip-5mm(l_1,\,l_2];\,(l_2,\,m_2];\,(m_2,\,l_3];\,(l_3,\,l_4];\,(l_4,\,m_3];\dots(m_k,\,l_{2k-1}];\,(l_{2k-1},\,l_{2k}];\nonumber\\
		&&\hskip-5mm(l_{2k},\,m_{k+1}];\dots\nonumber
	\end{eqnarray}
	and satisfying the conditions
	\begin{eqnarray}\label{dflength}
	&&l_{2k-1}-m_k=a>0,\quad
	m_{k+1}-l_{2k}=l_{2k}-l_{2k-1}=\nonumber\\
	&&\frac{1}{\sqrt{\alpha_k}}\Big[\pi/2-
	\arctan\big(\tanh(\sqrt{\alpha_k}a)\big)\Big].
	\end{eqnarray}
	The potential $V(x)$ is constructed in the following manner:
	\begin{eqnarray}\label{dfVx}
	V(x)=\left\{\begin{array}{ll}
	\alpha_k&\quad\mathrm{for}\quad
	x\in(m_k,\,l_{2k-1}],\\
	-\alpha_k&\quad\mathrm{for}\quad x\in(l_{2k-1},\,l_{2k}],\\
	-\alpha_{k+1}&\quad\mathrm{for}\quad x\in(l_{2k},\,m_{k+1}].
	\end{array}\right.
	\end{eqnarray}
	Consider the following semi-intervals, overlapping by pairs and
	covering the axis $\R$:
	\begin{equation*}
	(c_k,\,c_{k+1}]=(m_k,\,l_{2k-1}]\cup(l_{2k-1},\,l_{2k}]\cup(l_{2k},\,m_{k+1}]\cup(m_{k+1},\,l_{2k+1}],
	\end{equation*}
	and on each of them consider the boundary value problem for the
	equation $L\phi=0$ with the boundary conditions
	$\phi(c_k)=\phi(c_{k+1})=0$. It is easy to check that the solution
	of this problem is
	\begin{displaymath}
	\phi_k(x)==\left\{\begin{array}{ll}
	\frac{\sinh\big(\sqrt{\alpha_k}(x-c_k)\big)}{\sqrt{\sinh^2\big(\sqrt{\alpha_k}\,a\big)+\cosh^2\big(\sqrt{\alpha_k}\,a\big)}}&\quad\mathrm{for}\quad x\in(m_k,\,l_{2k-1}],\\
	\sin\Big[\sqrt{\alpha_k}\big(x-c_k-a\big)+\\\arctan\big(\tanh(\sqrt{\alpha_k}\,a)\big)\Big]&\quad\mathrm{for}\quad
	x\in(l_{2k-1},\,l_{2k}],\\
	\sin\Big[\sqrt{\alpha_{k+1}}\big(c_{k+1}-a-x\big)+\\\arctan\big(\tanh(\sqrt{\alpha_{k+1}}\,a)\big)\Big]&\quad\mathrm{for}\quad
	x\in(l_{2k},\,m_{k+1}],\\
	\frac{\sinh\big(\sqrt{\alpha_{k+1}}(c_{k+1}-x)\big)}{\sqrt{\sinh^2\big(\sqrt{\alpha_{k+1}}\,a\big)+\cosh^2\big(\sqrt{\alpha_{k+1}}\,a\big)}}&\quad\mathrm{for}\quad
	x\in(m_{k+1},\,l_{2k+1}].
	\end{array}\right.
	\end{displaymath}
	Since $\phi_k(x)\neq 0$ for $x\in(c_k,\,c_{k+1})$, then
	$\phi_k(x)$ is an eigenfunction corresponding to the minimal
	eigenvalue $\lambda_0=0$ of the boundary problem. Furthermore, it
	is easy to see that on each interval with the length $<a$ the
	equation $Ly=0$ has a solution which does not vanish. Hence by the
	localization principle the operator $L$ is bounded below and its
	spectrum is not discrete. Furthermore it is known that the
	spectrum $\sigma(H)$ of the operator $H$ is equal to the algebraic
	sum of the spectrum $\sigma(L)$ of the operator $L$ and $d-1$
	exemplars of the spectrum $\sigma(M)$ of the operator $M$:
	$\sigma(H)=\sigma(L)+\sigma(M)+\dots+\sigma(M)$ \cite{Ber}, hence the
	operator $H$ is bounded below and  its spectrum is not discrete.
	Let us show that the potential $W(\e)=V(x_1)+\sum_{j=2}^dU(x_j)$
	of the operator $H$ satisfies the condition \eqref{cndmolch1}. 
	Notice that by Proposition \eqref{thcondlebesgue} and Remark \ref{remwithoutnonneg}
	it is enough to show that for any $r>0$ and any
	sequence of points
	$\by_\nu=\langle y_{1,\nu},y_{2,\nu},\dots,y_{d,\nu}\rangle\in\R^d$, such that
	\begin{equation}\label{ynuinfty}
	\lim_{\nu\rightarrow\infty}|\by_\nu|=\infty,
	\end{equation}
	the condition	is valid:
	\begin{equation}\label{cndmeasinstcap1}
	\lim_{\nu\rightarrow\infty}\inf_{F\in\mathcal{M}_{\tilde\gamma(r)}(\by_\nu,r)}\int_{B_r(\by_\nu)\setminus
		F}W(\e)\, \mathrm{d}\e=\infty,
	\end{equation}
	where $\tilde\gamma(r)\in(0,1)$. 
Consider two cases:

1) $\lim_{\nu\rightarrow\infty}|y_{1,\nu}|=\infty$;

2) $\limsup_{\nu\rightarrow\infty}|y_{1,\nu}|<\infty$.

In the case 1) consider the following  sets in $\R^d$;
	\begin{equation}\label{dfEk}
	E_k:=(l_{2k-1},\,m_{k+1}]\times\R\times\dots\times\R.
	\end{equation}
	In view of \eqref{dflength}, for any $\nu$ the number $n_\nu$ of the
	sets $E_k\; (k\in (k_\nu,\;k_\nu+n_\nu-1) )$, intersecting the ball $B_r(\by_\nu)$, is finite and moreover: $N=\sup_{\,\nu\in\N}n_\nu<\infty$ and $k_\nu\rightarrow\infty$ as $\nu\rightarrow\infty$. Let as take
	$F\in\mathcal{M}_{\tilde\gamma(r)}(\by_\nu,r)$. Then
	\begin{equation*}
	\mathrm{mes}_d(B_r(\by_\nu)\setminus F)\ge
	(1-\tilde\gamma(r))\mathrm{mes}_d(B_r(0)).
	\end{equation*}
	 Taking into account
	\eqref{dflength}, \eqref{dfVx} and \eqref{dfEk} and the condition
	$U(x)>0$, we get:
	\begin{eqnarray}
		&&\int_{B_r(\by_\nu)\setminus F}W(\e)\,
		\mathrm{d}\e\ge\int_{(B_r(\by_\nu)\setminus F)\setminus\bigcup
			_{k=k_\nu}^{k_\nu+n_\nu-1}E_k}V(x_1)\, \mathrm{d}\e+\nonumber\\
		&&\int_{(B_r(\by_\nu)\setminus F)\cap\big(\bigcup
			_{k=k_\nu}^{k_\nu+n_\nu-1}E_k\big)}V(x_1)\, \mathrm{d}\e\ge\alpha_{k_\nu}\Big((1-\tilde\gamma(r))\mathrm{mes}_d(B_r(0))-\nonumber\\
		&&(2r)^{d-1}N\pi\frac{1}{\sqrt{\alpha_{k_\nu}}}\Big)-(2r)^{d-1}N\pi\sqrt{\alpha_{k_\nu}}.
		\nonumber
	\end{eqnarray}
	This estimate and \eqref{cndalfn} imply \eqref{cndmeasinstcap1}. In the case 2) 
	property \eqref{cndmeasinstcap1} follows easily from conditions \eqref{ynuinfty} and \eqref{Uxinfty}.
\end{example}

\begin{example}\label{ex3}
We will construct the example, promised in Remark \ref{remcndsembound}. We take as the domains $\mathcal{G}_{r_0}(\by)$ the balls $B_{r_0}(\by)$ and in each such a ball consider the nonhomogeneous $d^\prime$-Laplace equation \eqref{dLapleq}  with the Neumann boundary condition \eqref{neumboundcnd1} and a radial function $W^{(\by)}\e)=\tilde W^{(\by)}(|\e-\by|)$. Recall that $d^\prime=d/(d-1)$. First consider the boundary problem \eqref{dLapleq}-\eqref{neumboundcnd1} with $\by=0$. It has the form:
\begin{equation}\label{dLapleq1}
-\frac{d}{dr}\Big(\Big|\frac{du}{dr}\Big|^{d^\prime-2}\frac{du}{dr}\Big)-\frac{d-1}{r}\Big(\Big|\frac{du}{dr}\Big|^{d^\prime-2}\frac{du}{dr}\Big)=d^{1/(d-1)}\tilde W^{(0)}(r),
\end{equation}
\begin{equation}\label{dLapleq2}
\Big|\frac{du}{dr}\Big|^{d^\prime-2}\frac{du}{dr}\vert_{r=r_0}=0.
\end{equation}

  Then after the change of variable $v(r)=\Big|\frac{du}{dr}\Big|^{d^\prime-2}\frac{du}{dr}=(u^\prime(r))^{\langle d^\prime-1\rangle}$ equation \eqref{dLapleq1} becomes linear
 \begin{equation}\label{radODE}
 v^\prime+\frac{d-1}{r}v=-\tilde W^{(0)}(r)
 \end{equation}
 with the boundary condition 
 \begin{equation}\label{radbounond}
 v(r_0)=0.
 \end{equation}
 Here we denote $a^{\langle\alpha\rangle}=\mathrm{sign}(a)|a|^\alpha$. Since equation \eqref{radODE} is singular at the point $r=0$, we should add the regularity condition for $v(r)$ at this point. It is easy to check that the general solution of corresponding homogeneous equation is $z(r)=\frac{C}{r^{d-1}}$. Searching for solution of equation \eqref{radODE} in the form $v(r)=\frac{C(r)}{r^{d-1}}$, we obtain, that its general solution is $v(r)=\frac{C}{r^{d-1}}-\int_0^r\tilde W^{(0)}(\rho)\rho^{d-1}\,\mathrm{d}\rho$. We see that $v(r)$ is a regular solution of problem \eqref{radODE}-\eqref{radbounond} if and only if $C=0$ and
 \begin{equation*}
  \int_0^{r_0}\tilde W^{(0)}(\rho)\rho^{d-1}\,\mathrm{d}\rho=\frac{1}{\omega_d}\int_{B_{r_0}(0)}W^{(0)}(\e)\,\mathrm{d}\e=0.
 \end{equation*}
Hence we have the following solution of this problem:
  \begin{equation*}
  v(r)=
  -d^{1/(d-1)}\int_0^r\tilde W^{(0)}(\rho)\rho^{d-1}\,\mathrm{d}\rho
  \end{equation*}
 Therefore the derivative of solution of the problem  \eqref{dLapleq1}-\eqref{dLapleq2} has the form:
 \begin{equation}\label{formderiv}
 u^\prime(r)=\Big(-d^{1/(d-1)}\int_0^r\tilde W^{(0)}(\rho)\rho^{d-1}\,\mathrm{d}\rho\Big)^{\langle 1/(d^\prime-1)\rangle}
 \end{equation}
 
Now let us define the potential $V(\e)$. To this end we need a $\rho$-covering of $\R^d$  by balls. 
For each $\vec\ell\in\Z^d$ consider the ball $B_{r_0}(\vec\ell)$. It is easy to see that if
\begin{equation}\label{cndr01}
 r_0>\sqrt{d}/2,
\end{equation}
 these balls form a $s$-covering for some $s>0.$.

  Now assume that the potential $V(\e)$ has the form $V(\e)=V_0(\e)+W(\e)$,  where
\begin{equation}\label{V0sqrtmodx}
V_0(\e)=\sqrt{|\e|},
\end{equation}
\begin{equation*}
\mathrm{supp}(W)\subset\bigcup_{\vec\ell\in d\cdot\Z^d} B_{\rho_0}(\vec\ell),
\end{equation*}
$\rho_0<r_0$, $W(\e)=\tilde W^{(\vec\ell)}(|\e-\vec\ell|)$ for any  $\vec\ell\in d\cdot\Z^d$ and $\e\in B_{r_0}(\vec\ell)$,
 and each function $\tilde W^{(\vec\ell)}(\rho)$ has the property $\int_0^{r_0} \tilde W^{(\vec\ell)}(\rho)\rho^{d-1}\,\mathrm{d}\rho=0$, hence
 \begin{equation*}
 \mathrm{E}_{\vec\ell,\,r_0}(W)=0,
 \end{equation*}
 i.e.,
 \begin{equation}\label{EVeqEV0}
 \mathrm{E}_{\vec\ell,\,r_0}(V)=V_{0,\,\vec\ell},
 \end{equation}
 where we denote
 \begin{equation}\label{denotEV0}
 V_{0,\,\vec\ell}=\mathrm{E}_{\vec\ell,\,r_0}(V_0).
 \end{equation}
 It is easy to see that, in view of definition \eqref{V0sqrtmodx}, 
\begin{equation}\label{V0ltendinfty}
\lim_{|\vec\ell|\rightarrow\infty}V_{0,\,\vec\ell}=\infty
\end{equation}
and
\begin{equation}\label{devV-0}
\lim_{|\vec\ell|\rightarrow\infty}\;\sup_{\e\in B_{r_0}((\vec\ell)|}|V_0(\e)-V_{0,\,\vec\ell}|=0.
\end{equation}
Then in view of the result \eqref{formderiv} and the equality 
\begin{equation*}
V(\e)-\mathrm{E}_{\vec\ell,\,r_0}(V)=V_0(\e)-V_{0,\,\vec\ell}+W(\e)
\end{equation*}
condition \eqref{cndpert2} is equivalent to
\begin{equation}\label{cndpert3}
\limsup_{|\vec\ell|\rightarrow\infty}\Big(\int_0^{r_0}\Big|\int_0^r \tilde W^{(\vec\ell)}(\rho)\rho^{d-1}\,\mathrm{d}\rho \Big|^d\,\mathrm{d}r\Big)^{1/d}<
\frac{1}{C(d)},
\end{equation}
If this condition is fulfilled, then in view of \eqref{EVeqEV0}-\eqref{V0ltendinfty} and Corollary \ref{thchicpert}, the spectrum of the operator $H=-\Delta+V(\e)\cdot$ is discrete and bounded below. Let us choose  $\tilde W^{(\ell)}$ satisfying \eqref{cndpert3},  so that for $V(\e)$  condition \eqref{condbndblc} of Theorem \ref{thgenermolch1} is not fulfilled.
 We put:
\begin{eqnarray}\label{defWell}
\tilde W^{(\vec\ell)}(\rho)=\left\{\begin{array}{ll}
0&\mathrm{for}\quad \rho\in[0,\, 2\pi\psi_{\vec\ell})\cup(\rho_0,\, r_0]\\\\
\frac{A_{\vec\ell}\,\mathrm{sign}(\sin(\psi_{\vec\ell}\rho))}{\rho^{d-1}}&\mathrm{for}\quad \rho\in[2\pi/\psi_{\vec\ell}\,,\rho_0],
\end{array}
\right.
\end{eqnarray}
where $\psi_{\vec\ell}=2\pi|\vec\ell|/\rho_0$ and  a positive sequence $A_{\vec\ell}$ satisfies the condition 
\begin{equation}\label{Alinfty}
\lim_{|\vec\ell|\rightarrow\infty}A_{\vec\ell}=\infty .
\end{equation}
It will be chosen in what follows.
 In order such a definition of $W(\e)$ on the balls $B_{r_0}(\vec\ell)$ will be consistent at their intersections we need the condition
\begin{equation}
\forall\; \vec\ell_0\in \Z^d:\quad B_{\rho_0}(\vec\ell_0)\cap\bigcup_{\vec\ell\in\Z^d\setminus\{\vec\ell_0\}} B_{r_0}(\vec\ell)=\emptyset, 
\end{equation}
which is fulfilled if  $\sqrt{d}/2<1$, i.e. $d<4$, and 
\begin{equation}\label{cndr92}
\rho_0<1-\sqrt{d}/2.
\end{equation}
Hence we have to assume that $d=3$\footnote{Perhaps a more complicated construction of the desired covering of $\R^d$ by balls is also possible for $d>3$, but in this example we restrict ourselves to this special case.}.

Let us calculate
\begin{eqnarray}
&&\hskip-8mm\int_0^r \tilde W^{(\vec\ell)}(\rho)\rho^{d-1}\,\mathrm{d}\rho=A_{\vec\ell}\int_{2\pi/\psi_{\vec\ell}}^{r}
\mathrm{sign}(\sin(\psi_{\vec\ell}\rho))\,\mathrm{d}\rho=\nonumber\\
&&\hskip-8mm\frac{A_{\vec\ell}}{\psi_{\vec\ell}}\int_{2\pi}^{2\pi|\vec\ell| r/\rho_0}\mathrm{sign}(\sin t)\,\mathrm{d}t=
\frac{A_{\vec\ell}}{\psi_{\vec\ell}}\int_{2\pi[|\vec\ell| r/\rho_0]}^{2\pi|\vec\ell| r/\rho_0}
\mathrm{sign}(\sin t)\,\mathrm{d}t\nonumber.
\end{eqnarray}
Therefore
\begin{equation*}
\Big|\int_0^r \tilde W^{(\vec\ell)}(\rho)\rho^{d-1}\,\mathrm{d}\rho\Big|\le \frac{A_{\vec\ell}\,\rho_0}{|\vec\ell|}.
\end{equation*}
This means that if the condition
\begin{equation}\label{cxndAell}
\lim_{|\vec\ell|\rightarrow\infty}\frac{A_{\vec\ell}}{|\vec\ell|}=0
\end{equation}
is satisfied, then condition \eqref{cndpert3} is fulfilled. Let us put
\begin{equation}\label{defVell}
A_{\vec\ell}=V_{0,\,\vec\ell}\rho_0^{d-1}.
\end{equation}
In view of \eqref{devV-0}, it is enough to show that $V_{0,\,\vec\ell}+\tilde W^{(\vec\ell)}$ does not satisfy condition \eqref{condbndblc}. To this end let us find the function $(V_{0,\,\vec\ell}+\tilde W^{(\vec\ell)})_-$. Taking into account \eqref{defWell}, consider the sets
\begin{eqnarray}
&&\mathcal{E}_{\vec\ell}=\{\rho\in[0,\,\rho_0]:\;V_{0,\,\vec\ell}+\tilde W^{(\vec\ell)}(\rho) <0 \}=\nonumber\\
&&\Big\{\rho\in[0,\,\rho_0]:\frac{\mathrm{sign}(\sin(\psi_{\vec\ell}\rho))}{\rho^{d-1}} <-\frac{V_{0,\,\vec\ell}}{A_{\vec\ell}} \Big\}\supseteq
\nonumber\\
&&\bigcup_{k=2}^{\Theta(\vec\ell)}\Big(\frac{\rho_0}{|\vec\ell|}(k-1/2),\,\frac{\rho_0}{|\vec\ell|}k\Big),
\nonumber
\end{eqnarray}
where $\Theta(\vec\ell)=[|\vec\ell|]$.   
Then  in view of \eqref{defVell},
\begin{eqnarray}
&&\int_0^{r_0}(V_{0,\,\vec\ell}+\tilde W^{(\vec\ell)}(\rho))_-^{d/2}\rho^{d-1}\,\mathrm{d}\rho=
-\int_{\mathcal{E}_{\vec\ell}}(V_{0,\,\vec\ell}+\tilde W^{(\vec\ell)}(\rho))^{d/2}\rho^{d-1}\,\mathrm{d}\rho\ge\nonumber\\
&&A_{\vec\ell}^{d/2}\sum_{k=2}^{\Theta(\vec\ell)}\int_{\frac{\rho_0}{|\vec\ell|}(k-1/2}^{\frac{\rho_0}{|\vec\ell|}k}
\Big(\frac{1}{\rho^{d-1}}-\frac{1}{\rho_0^{d-1}}\Big)^{d/2}\rho^{d-1}\,\mathrm{d}\rho\ge
\nonumber\\
&&\frac{A_{\vec\ell}^{d/2}}{\rho_0^{(d-1)(d/2-1)}}\sum_{k=2}^{\Theta(\vec\ell)}\int_{\frac{\rho_0}{|\vec\ell|}(k-1/2)}^{\frac{\rho_0}{|\vec\ell|}k}
\Big(1-\rho^{d-1}/\rho_0^{d-1}\Big)^{d/2}\,\mathrm{d}\rho\ge
\nonumber\\
&&
\frac{A_{\vec\ell}^{d/2}\rho_0}{2\rho_0^{(d-1)(d/2-1)}}\sum_{k=2}^{\Theta(\vec\ell)}\Big(
\Big({1-\Big(\frac{k}{|\vec\ell|}}\Big)^{d-1}\Big)^{d/2}\frac{1}{|\vec\ell|}.
\nonumber
\end{eqnarray}
Since the last sum tends to 
\begin{equation*}
\int_0^{1}\Big(1-t^{d-1}\Big)^{d/2}\,\mathrm{d}t
\end{equation*}
as $|\vec\ell|\rightarrow\infty$, then in view of \eqref{Alinfty} 
\begin{equation*}
\lim_{|\vec\ell|\rightarrow\infty}\int_0^{r_0}(V_{0,\,\vec\ell}+\tilde W^{(\vec\ell)}(\rho))_-^{d/2}\rho^{d-1}\,\mathrm{d}\rho=\infty.
\end{equation*}
Thus, condition \eqref{condbndblc} is not satisfied for the potential $V(\e)$.
\end{example}
	
\appendix

\section{Divergence constrained transportation problem}
\label{sec:transprob}

In \cite{Br-Pet}, \cite{Br-Car-San} the following optimization problem is considered. Let
$\Omega\subset\R^d$ be a bounded domain with a smooth boundary and $W\in L_1(\Omega)$. 
Let $H:\;\R^d\rightarrow\R$ be a function, satisfying the conditions:

(i) $H$ is strictly convex radially symmetric function, with $H(0)=0$ and it is differentiable in $\R^d\setminus\{0\}$;

(ii) for some $p\in(1,\infty)$ and positive constants $a,\;b$
\begin{equation*}
a|\sigma|^p\le H(\sigma)\le b(|\sigma|^p+1);
\end{equation*}

(iii) for some positive constant $c$
\begin{equation*}
|\nabla H(\sigma)|\le c(|\sigma|^{p-1}+1),\quad \sigma\in\R^d\setminus\{0\}.  
\end{equation*}
Denote by $\mathcal{D}_p(\Omega,W)$ the set of all vector fields $\vec F\in L_p(\Omega,\R^d)$ on
$\Omega$ satisfying (in the generalized sense) the equation
$\mathrm{div}\,\vec F=W$ and the boundary condition $\vec F(x)\cdot\vec n(x)\vert_{x\in\partial \Omega}=0$. 
Consider the quantity
\begin{equation}\label{optprobtrans}
D_p(\Omega,W)=\inf_{\vec F\in\mathcal{D}_p(\Omega,W)}  H(\vec F).
\end{equation}
Let $H^\star$ be Legendre transformation of $H$.
In \cite[Theorem 2.1]{Br-Car-San}, \cite[Remark 3.5]{Br-Pet} the following claim
was established:
\begin{proposition}\label{prtranspr}
	If
	\begin{equation*}
	\int_\Omega W(\e)\,\mathrm{d}\e=0, 
	\end{equation*}
	and $q=p/(p-1)$, then there exists a unique minimizer for \eqref{optprobtrans}, having the form
	\begin{equation*}
	\vec F=\nabla H^\star(\nabla u),
	\end{equation*}
	where $u\in W_q^1(\Omega)$ is a weak solution of the nonhomogeneous $q$-Laplace equation
	\begin{equation}\label{pLapleq}
	\mathrm{div}\big(\nabla H^\star(\nabla u)\big)=W(\e)
	\end{equation}
	which satisfies the Neumann boundary condition
	\begin{equation}\label{neumboundcnd2}
	\nabla H^\star(\nabla u)\cdot\vec n \vert_{\partial\Omega M}=0
	\end{equation}
and this solution is unique (up to a constant summand) in $W_q^1(\Omega)$.
\end{proposition}

\begin{remark}\label{remtrasport}
In the case where $H(\sigma)=|\sigma|^p$ 
\begin{equation*}
H^\star(\sigma)=(1-1/p)(1/p)^{p-1} |\sigma|^q,
\end{equation*}
hence
\begin{equation*}
\nabla H^\star(\sigma)=\frac{1}{p^{1/(p-1)}}|\sigma|^{q-2}\sigma.
\end{equation*}
Therefore the minimizer for \eqref{optprobtrans} is
\begin{equation*}
\vec F=\frac{1}{p^{1/(p-1)}}|\nabla u|^{q-2}\nabla u
\end{equation*}
and the problem  \eqref{pLapleq}-\eqref{neumboundcnd2} has the form:
\begin{equation*}
\mathrm{div}(|\nabla u|^{q-2}\nabla u)=p^{1/(p-1)}W(\e),\quad |\nabla u|^{q-2}\nabla u\cdot\vec n\vert_{\partial\Omega}=0.
\end{equation*}
\end{remark}
	
	\section{Integral inequalities of Riccati type}\label{sec:intineqRicc} 

		\subsection{\hskip2mm Formulation of the claim}
	\label{subsec:appendformul}
	
	Let $\{u_n(t)\}_{n=1}^\infty$ be a sequence of continuous
	real-valued functions, defined on $(0,r)\,(r>0)$. Consider the
	following sequence of subsets of the triangle
	$\Delta_r=\{\langle s,t\rangle\in\R^2:\;0< s<t< r)\}$:
	\begin{eqnarray}\label{defEn}
	E_n=\{\langle s,t\rangle\in\Delta_r:\;u_n(t)-u_n(s)>\int\limits_s\limits^tu_n^2(\xi)\,
	d\xi+ \lambda_n(t-s)\},
	\end{eqnarray}
	where $\{\lambda_n\}_{n=1}^\infty$ is a sequence of real numbers.
	
	The main claim of this section is following:
	\begin{proposition}\label{mainprAp2}
		If $\lambda_n\rightarrow+\infty$ for $n\rightarrow\infty$, then
		$\lim\limits_{n\rightarrow\infty}\mathrm{mes}_2E_n=0$.
	\end{proposition}

	\subsection{\hskip2mm Lemma on integral inequalities}
	\label{subsec:mainlemintineq}
	
	The claim formulated above is based on the following lemma:
	\begin{lemma}\label{lmintineq}
		Let  $\{u_n(t)\}_{n=1}^\infty$ be a sequence of measurable bounded
		real-valued functions such that each of them is defined on
		$[t_n,r]\,(0\le t_n<r)$, and $X_n\subseteq [t_n,r]$ be a sequence of
		measurable sets. Suppose that for some sequence of real numbers
		$\lambda_n\rightarrow+\infty$ for $n\rightarrow\infty$ the
		inequalities are valid
		\begin{equation}\label{intineq}
		u_n(t)\ge\int_{t_n}^tu_n^2(\xi)\, \mathrm{d}\xi+
		\lambda_n(t-t_n)\quad \mathrm{for}\quad t\in X_n\quad (n=1,2,\dots).
		\end{equation}
		Then $\mathrm{mes}_1 X_n\to 0$ for $n\rightarrow\infty$.
	\end{lemma}
	\begin{proof}
		Assume on the contrary that for some subsequence $X_n\;$
		$\mathrm{mes}_1 X_n\ge c_1>0$. Then there exists a sequence of
		closed sets $X_n^\prime\subseteq X_n$ such that $\mathrm{mes}_1
		X_n^\prime\ge c_2>0$. Let $\chi_n(t)$ be the characteristic function
		of $X_n^\prime$. Consider the function
		$z_n(t)=\int_{t_n}^t\chi_n(\xi)\,\mathrm{d}\xi$. Let
		$X_n^{\prime\prime}$ be the set of points of growth of $z_n(t)$. It
		is clear that $X_n^{\prime\prime}$ is a closed perfect set and
		$X_n^{\prime\prime}\subseteq X_n^\prime$. Removing from the set
		$X_n^{\prime\prime}$ the left endpoints of the open intervals
		adjacent to it, we get a set $X_n^{\prime\prime\prime}\subset
		X_n^{\prime\prime}$ with
		$\mathrm{mes}_1X_n^{\prime\prime\prime}=\mathrm{mes}_1X_n^{\prime\prime}$.
		It is easy to see that the function $z_n(t)$ maps bijectively the
		set $X_n^{\prime\prime\prime}$ on an interval $[0,b_n)$ with
		$b_n=\mathrm{mes}_1X_n^{\prime\prime\prime}>c_2>0$. Hence on
		$X_n^{\prime\prime\prime}$ the function $z=z_n(t)$ has inverse
		$t=t_n(z)$. Let us choose $N_1>0$ such that $\lambda_n\ge 0$ for
		$n\ge N_1$. Then we get from (\ref{intineq}) for $n\ge N_1$:
		\begin{equation*}
		u_n(t)\ge\int_{t_n}^tu_n^2(\xi)\chi_n(\xi)\,\mathrm{d}\xi+
		\lambda_n\int_{t_n}^t\chi_n(\xi)\,\mathrm{d}\xi=\int_{t_n}^tu_n^2(\xi)\,\mathrm{d}z_n(\xi)+
		\lambda_n z_n(t)
		\end{equation*}
		for $t\in X_n^{\prime\prime\prime}$. Performing the change of
		variables $z=z_n(t)$, $v_n(z)=u_n(t_n(z))$, we get:
		\begin{equation}\label{afterchangvar}
		v_n(z)\ge\int_0^zv_n^2(\eta)\,\mathrm{d}\eta+\lambda_nz\quad\mathrm{for}\quad
		0\le z<b_n.
		\end{equation}
		As it is easy to see, the function $v_n(z)$ is measurable and
		bounded. Furthermore, (\ref{afterchangvar}) implies that there is
		$N_2\ge N_1$ such that $v_n(z)\ge 0$ on $[0,b_n)$ for any $n\ge
		N_2$. Consider the integral equation
		\begin{equation}\label{Riccinteq}
		w_n(z)=\int_0^zw_n^2(\eta)\,\mathrm{d}\eta+\lambda_nz,
		\end{equation}
		whose solution is $w_n(z)=\sqrt{\lambda_n}\tan(\sqrt{\lambda_n}z)$.
		We see that $w_n(z)\to+\infty$ for $z\uparrow
		z_n=\frac{\pi}{2\sqrt{\lambda_n}}$ and $w_n(z)\ge 0$ on $[0,z_n]$.
		Since $z_n\to 0$ for $n\to\infty$ and $b_n>c_2>0$, then there is
		$N_3\ge N_2$ such that $[0,z_n]\subset [0,b_n)$ for any $n\ge N_3$.
		Subtracting the equality~\eqref{Riccinteq} from the
		inequality~\eqref{afterchangvar}, we get the inequality
		\begin{equation*}
		v_n(z)-w_n(z)\ge\int_0^z(v_n(\eta)+w_n(\eta))(v_n(\eta)-w_n(\eta))\,\mathrm{d}\eta,\quad
		z\in[0,z_n],
		\end{equation*}
		which implies, in view of Proposition~\ref{prlinintineq}, that
		$v_n(z)\ge w_n(z)$ for $z\in[0,z_n]$. The last inequality
		contradicts the boundedness of $v_n(z)$.
	\end{proof}
	
	In the proof of the previous lemma we have used the following well
	known claim on linear integral inequalities (see \cite{Az-Ts}).
	
	\begin{proposition}\label{prlinintineq}
		Let $u(t)$ be a measurable bounded real-valued function, defined on
		$[0,a]\,(a>0)$, and $K(t,\xi)$ be a measurable bounded function,
		defined on the triangle $\Delta=\{(\xi,t)\in\R^2\,:\,0\le\xi\le t\le
		a\}$ such that $K(t,\xi)\ge 0$ on $\Delta$. Then the integral
		inequality
		\begin{equation*}
		u(t)\ge\int_0^tK(t,\xi)u(\xi)\,\mathrm{d}\xi,\quad t\in[0,a]
		\end{equation*}
	implies that $u(t)\ge 0$ on $[0,a]$.
	\end{proposition}
	
	\subsection{\hskip2mm Splitting the sets $E_n$ and investigation of their parts}
	\label{subsubsec:splitEn}
	
	We will use the following notation.\vskip2mm
	
	\noindent$P_1$ and $P_2$ are the projection
	mappings of the cartesian product $\R^2=\R\times\R$ on the first and
	the second factor respectively;\vskip1mm
	
	\noindent If $E$ is a subset of $\R^2$, then
	$S_1(E)\langle s\rangle$ and $S_2(E)\langle t\rangle$ are the sections
	of $E$ through the points $s\in P_1(E)$ and
	$t\in P_2(E)$, i.e.,
	\begin{equation*}
	S_1(E)\langle s\rangle=\{t\in\R\,:\,\langle s,t\rangle\in E\}
	\end{equation*}
	 and
	 \begin{equation*}
	 S_2(E)\langle t\rangle=\{s\in\R\,:\,\langle s,t\rangle\in E\}.
	 \end{equation*}
	 If the set $E$ is measurable and $\mathrm{mes}_2E<\infty$,
	then by Fubini's Theorem, $S_1(E)\langle s\rangle$ are measurable
	for almost all $s$, $S_2(E)\langle s\rangle$ are measurable for
	almost all $t$ and
	\begin{equation}\label{Fubini}
	\mathrm{mes}_2E=\int\limits_{P_1E}\mathrm{mes}_1\,S_1(E)\langle s\rangle\,\mathrm{d}s=
	\int\limits_{P_2E}\mathrm{mes}_1\,S_2(E)\langle t\rangle\,\mathrm{d}t.
	\end{equation}
	\vskip2mm
	
	Let us return to the sequence of continuous real-valued functions
	$u_n(t)$ $(t\in[0,r])$  and the sets
	$E_n\subseteq\Delta_r=\{\langle s,t\rangle\in\R^2:\,0< s< t<r)\}$, defined by
	\eqref{defEn} (see Section~\ref{subsec:appendformul}). Denote
	\begin{equation}\label{dfepsplmin}
	\epsilon_n^+=\{t\in(0,r)\,:\,u_n(t)>0\},\quad
	\epsilon_n^-=(0,r)\setminus\epsilon_n^+,
	\end{equation}
	\begin{equation}\label{dfEnplmin}
	E_n^\prime=((0,r)\times\epsilon_n^-)\cap E_n,\quad
	E_n^{\prime\prime}=((0,{}r)\times\epsilon_n^+)\cap E_n.
	\end{equation}
	It is clear that
	\begin{equation}\label{splitEn1}
	E_n=E_n^\prime\cup E_n^{\prime\prime}
	\end{equation}
	\begin{lemma}\label{treatEnprime}
		If $\lambda_n\to\infty$ for $n\to\infty$, then
		$\lim\limits_{n\to\infty}\mathrm{mes}_2E_n^\prime=0$.
	\end{lemma}
	\begin{proof}
		In view of~\eqref{Fubini}, it is sufficient to prove that
		$\mathrm{mes}_1\,S_2(E_n^\prime)\langle t\rangle \to 0$ for
		$n\to\infty$ uniformly with respect to
		$t\in P_2E_n^\prime=\epsilon_n^-\cap P_2E_n$.
		Assume on the contrary that for some subsequence of $E_n^\prime$
		there is a sequence $t_n\in P_2E_n^\prime$ such that
		$\mathrm{mes}_1\,S_2(E_n^\prime)\langle t_n\rangle>c>0$. In view of
		\eqref{defEn}, we have:
		\begin{equation*}
		u_n(t_n)-u_n(s)>\int\limits_s\limits^{t_n}u_n^2(\xi)\,
		\mathrm{d}\xi+ \lambda_n(t_n-s)\quad\mathrm{for}\quad
		s\in S_2(E_n^\prime)\langle t_n\rangle.
		\end{equation*}
		Deriving the change of the variables $z=t_n-s$, $v_n(z)=-u_n(t_n-z)$
		and taking into account that $u_n(t_n)\le 0$ (because
		$t_n\in\epsilon_n^-$), we get:
		\begin{equation*}
		v_n(z)>\int\limits_0\limits^zv_n^2(\xi)\,
		\mathrm{d}\xi+\lambda_nz,\quad z\in Z_n,
		\end{equation*}
		where $\mathrm{mes}_1Z_n>c>0$. The latter contradicts the claim of
		Lemma~\ref{lmintineq}.
	\end{proof}
	
	\begin{remark}\label{remcuts}
		Since in view of~\eqref{dfEnplmin}$_1$
		$S_2(E_n^\prime)\langle t\rangle=S_2(E_n)\langle t\rangle$ for
		$t\in P_2E_n^\prime$, then arguments of the proof of
		Lemma~\ref{treatEnprime} imply that
		$\mathrm{mes}_1\,S_2(E_n)\langle t\rangle \to 0$ for $n\to\infty$
		uniformly with respect to $t\in P_2E_n^\prime$.
	\end{remark}
	
	Now  we will treat the sets $E_n^{\prime\prime}$, defined by
	\eqref{dfEnplmin}$_2$. We will split them into two parts. Let us
	consider the sections $S_1(E_n)\langle s\rangle$ for
	$s\in P_1E_n$. Since  $S_1(E_n)\langle s\rangle$ and
	$\epsilon_n^+$ (defined by \eqref{dfepsplmin}$_1$) are open sets,
	then
	\begin{equation*}
		S_1(E_n)\langle s\rangle=\bigcup\limits_{i=1}\limits^\infty(p_i^{(n)}(s),\,q_i^{(n)}(s))
	\end{equation*}
and $\epsilon_n^+=\bigcup\limits_{j=1}\limits^\infty (s_j^{(n)},\,
	r_j^{(n)})$, where the intervals, taking part in each of these
	representations, are disjoint. Then in view
	of~\eqref{dfEnplmin}$_2$,   we have for
	$s\in P_1E_n^{\prime\prime}\langle s\rangle$:
	\begin{eqnarray}\label{rprEprpr}
	&&S_1(E_n^{\prime\prime})\langle s0\rangle=\bigcup\limits_{i,j=1}\limits^\infty(s_j^{(n)},\,
	r_j^{(n)})\cap
	(p_i^{(n)}(s),\,q_i^{(n)}(s))=\nonumber\\
	&&\bigcup\limits_{k=1}\limits^\infty(l_k^{(n)}(s),\,m_k^{(n)}(s)),
	\end{eqnarray}
	where $(l_k^{(n)}(s),\,m_k^{(n)}(s))=(s_j^{(n)},\, r_j^{(n)})\cap
	(p_i^{(n)}(s),\,q_i^{(n)}(s))$ for some $i,\,j$. It is easy to see
	that
	\begin{equation}\label{equiv}
	l_k^{(n)}(s)\in S_1(E_n)\langle s\rangle \Longrightarrow
	l_k^{(n)}(s)=s_j^{(n)}.
	\end{equation}
	If $s_j^{(n)}\in P_2E_n$, consider the set
	\begin{equation}\label{dfkappaj}
	\kappa_j^{(n)}=\big(S_2(E_n)\langle s_j^{(n)}\rangle\times(s_j^{(n)},\,
	r_j^{(n)})\big)\cap E_n.
	\end{equation}
	It is clear that the sets $\kappa_j^{(n)}$ are open and
	$\kappa_j^{(n)}\subseteq E_n^{\prime\prime}$.	Let us define the sets
	\begin{equation}\label{dfEnstar}
	E_n^\star=\bigcup\limits_{j=1}\limits^\infty\kappa_j^{(n)}.
	\end{equation}
	Then $E_n^\star\subseteq E_n^{\prime\prime}$ and
	$S_1(E_n^\star)\langle s\rangle=\bigcup\limits_{j=1}\limits^\infty S_1(\kappa_j^{(n)})\langle s\rangle$
	for $s\in P_1E_n^\star$, where
	\begin{equation*}
	S_1(\kappa_j^{(n)})\langle s\rangle=(s_j^{(n)},\,
	r_j^{(n)})\cap S_1(E_n)\langle s\rangle
	\end{equation*}
	 for
	$s\in P_1\kappa_j^{(n)}=S_2()E_n\langle s_j^{(n)}\rangle$.
	The last equality is equivalent to
	\begin{equation*}
	s_j^{(n)}\in S_1(E_n)\langle s\rangle.
	\end{equation*}
	 Hence  we have:
	\begin{eqnarray}\label{cutEnst}
	&&S_1(E_n^\star)\langle s\rangle=\bigcup\limits_{j\,:\,s_j^{(n)}\in S_1(E_n)\langle s\rangle}(s_j^{(n)},\,
	r_j^{(n)})\cap S_1E_n\langle s\rangle=\nonumber\\
	&&\bigcup\limits_{i=1}\limits^\infty\bigcup\limits_{j\,:\,s_j^{(n)}\in S_1(E_n)\langle s\rangle}(s_j^{(n)},\,
	r_j^{(n)})\cap(p_i^{(n)}(s),\,q_i^{(n)}(s)).
	\end{eqnarray}
	Let us define the sets $E_n^{\star\star}=E_n^{\prime\prime}\setminus
	E_n^\star$. Then
	\begin{equation*}
	S_1(E_n^{\star\star})\langle s\rangle=S_1(E_n^{\prime\prime})\langle s\rangle\setminus
	S_1(E_n^\star)\langle s\rangle.
	\end{equation*}
	 Taking into account
	\eqref{rprEprpr}, \eqref{cutEnst} and \eqref{equiv},  we get
	\begin{equation}\label{cutEnstst}
 S_1(E_n^{\star\star})\langle  s\rangle =\bigcup\limits_{i=1}\limits^\infty(l_{k_i}^{(n)}(s),\,m_{k_i}^{(n)}(s)),
	\end{equation}
	where	
	\begin{equation}\label{lknotin}
	l_{k_i}^{(n)}(s)\notin S_1(E_n)\langle s\rangle
	\end{equation}
		Here $(l_{k_i}^{(n)}(s),\,m_{k_i}^{(n)}(s))$ is a subsequence of
	intervals, forming the set\\
	$S_1(E_n^{\prime\prime})\langle s\rangle$. Taking into account
	\eqref{splitEn1},  we get:
	\begin{equation}\label{splitEn2}
	E_n=E_n^\prime\cup E_n^\star\cup E_n^{\star\star}.
	\end{equation}
	
	\begin{lemma}\label{treatEnst}
		If $\lambda_n\to\infty$ for $n\to\infty$, then
		$\lim\limits_{n\to\infty}\mathrm{mes}_2E_n^\star=0$.
	\end{lemma}
	\begin{proof}
		Since the intervals $(s_j^{(n)},\, r_j^{(n)})$, forming the set
		$\epsilon_n^+$ (defined by~\eqref{dfepsplmin}$_1$), are disjoint,  we
		have from \eqref{dfEnstar} and \eqref{dfkappaj} that for
		$t\in P_2\kappa_j^n\;$
		$S_2(E_n^\star)\langle t\rangle=S_2(\kappa_j^n)\langle t\rangle$ and
		$S_2(\kappa_j^n)\langle t\rangle =S_2(E_n)\langle s_j^n\rangle\cap S_2(E_n)\langle t\rangle
		\subseteq S_2(E_n)\langle s_j^n\rangle$. Thus,
		\begin{equation}\label{cutsinclus}
		\forall\;t\in P_2\kappa_j^n:\quad S_2(E_n^\star)\langle t\rangle\subseteq S_2(E_n)\langle s_j^n\rangle.
		\end{equation}
		Since $s_j^n\in\epsilon_n^-$, then by Remark \ref{remcuts},
		$\mathrm{mes}_1S_2(E_n)\langle s_j^n\rangle\to 0$ for $n\to\infty$
		uniformly with respect to $j$. Then inclusion \eqref{cutsinclus} and
		equality
		$P_2E_n^\star=\bigcup\limits_{j=1}\limits^\infty P_2\kappa_j^n$
		imply that $\mathrm{mes}_1 S_2(E_n^\star)\langle t\rangle\to 0$ for
		$n\to\infty$ uniformly with respect to $t\in P_2E_n^\star$.
		This circumstance and equality \eqref{Fubini} imply the desired
		claim.
	\end{proof}
	
	\begin{lemma}\label{treatEnstst}
		If $\lambda_n\to\infty$ for $n\to\infty$, then
		$\lim\limits_{n\to\infty}\mathrm{mes}_2E_n^{\star\star}=0$.
	\end{lemma}
	\begin{proof}
		We should prove that
		$\mathrm{mes}_1 S_1(E_n^{\star\star})\langle s\rangle\to 0$ for
		$n\to\infty$ uniformly with respect to
		$s\in P_1E_n^{\star\star}$. Assume on the contrary that for
		some subsequence $E_n^{\star\star}$ there is a sequence
		$s_n\in P_1E_n^{\star\star}$ such that
		$\mathrm{mes}_1 S_1(E_n^{\star\star})\langle s_n\rangle >c_1>0$.
		According to \eqref{defEn},  we have for $t\in
		S_1(E_n^{\star\star})\langle s_n\rangle$:
		\begin{equation}\label{intineqcutEnstst}
		u_n(t)-u_n(s_n)>\int\limits_{s_n}\limits^tu_n^2(\xi)\,
		\mathrm{d}\xi+ \lambda_n(t-s_n)
		\end{equation}
		Consider two cases:\vskip1mm
		
		1) for some subsequence $s_n\;$ $u_n(s_n)\ge 0$;\vskip1mm
		
		2) there is $N_1>0$ such that $u_n(s_n)<0$ for $n\ge N_1$.\vskip1mm
		
		In the case 1) \eqref{intineqcutEnstst} implies the inequalities
		\begin{equation*}
		u_n(t)>\int\limits_{s_n}\limits^tu_n^2(\xi)\, \mathrm{d}\xi+
		\lambda_n(t-s_n),\quad t\in
		S_1(E_n^{\star\star})\langle s_n\rangle,
		\end{equation*}
		which contradict the claim of Lemma \ref{lmintineq}.
		
		We now turn to the case 2). Assume that $n\ge N_1$. In view of
		\eqref{cutEnstst}-\eqref{lknotin},
		$S_1(E_n^{\star\star})\langle s_n\rangle=\bigcup\limits_{k=1}\limits^\infty(l_k^{(n)},\,m_k^{(n)})$
		with $l_k^{(n)}\notin S_1(E_n)\langle s_n\rangle$. Since the
		functions $u_n(t)$ are continuous, $u_n(s_n)<0$ and
		$(l_k^{(n)},\,m_k^{(n)})\subseteq\epsilon_n^+\cap(s_n,\,r)$, then
		$l_k^{(n)}>s_n$. Denote $\mu_n=\inf
		S_1(E_n^{\star\star})\langle s_n\rangle$. For any $n$ let me choose
		one of the left endpoints $\tilde l_n=l_{k_n}^{(n)}$ of the
		intervals forming $S_1(E_n^{\star\star})\langle s_n\rangle$ such
		that $\lim\limits_{n\to\infty}(\tilde l_n-\mu_n)=0$. Denote
		$X_n=S_1(E_n^{\star\star})\langle s_n\rangle\cap (\tilde l_n,\,r)$.
		Observe that $(\tilde l_n,\,\tilde m_n)$ with $\tilde
		m_n=m_{k_n}^{(n)}$ is the very left interval among the intervals
		forming $X_n$. Then there is $N_2\ge N_1$ such that
		$\mathrm{mes}_1X_n>c_2>0$ for $n\ge N_2$. Consider the continuous
		function
		\begin{equation*}
		\alpha_n(t)=u_n(t)-u_n(s_n)-\int\limits_{s_n}\limits^tu_n^2(\xi)\,
		d\xi-\lambda_n(t-s_n).
		\end{equation*}
		Since $\tilde l_n\notin S_1(E_n)\langle s_n\rangle$ and $\tilde
		l_n>s_n$, i.e.,$(s_n,\,\tilde l_n)\notin E_n$ and $()s_n,\,\tilde
		l_n\>\in\Delta_r$, then in view of \eqref{defEn}, $\alpha_n(\tilde
		l_n)\le 0$. On the other hand, in view of \eqref{intineqcutEnstst},
		$\alpha_n(t)>0$ for $t\in(\tilde l_n,\,\tilde m_n)$. Hence
		$\alpha_n(\tilde l_n)=0$, i.e.,
		\begin{equation}\label{equality}
		u_n(\tilde l_n)-u_n(s_n)=\int\limits_{s_n}\limits^{\tilde
			l_n}u_n^2(\xi)\, \mathrm{d}\xi+\lambda_n(\tilde l_n-s_n).
		\end{equation}
		Observe that since $u_n(t)>0$ for $t\in(\tilde l_n,\,\tilde m_n)$,
		then $u_n(\tilde l_n)\ge 0$. Taking into account this fact and
		subtracting \eqref{equality} from \eqref{intineqcutEnstst},  we get
		the inequalities
		\begin{equation*}
		u_n(t)>\int\limits_{\tilde l_n}\limits^tu_n^2(\xi)\,
		\mathrm{d}\xi+ \lambda_n(t-\tilde l_n),\quad t\in X_n.
		\end{equation*}
		Again  we get the contradiction with the claim of Lemma
		\ref{lmintineq}. 
	\end{proof}
	
	\subsection{\hskip2mm Proof of Proposition \ref{mainprAp2}}
	\label{subsec:proofprAp2}
	\begin{proof}
		The claim follows immediately from representation \eqref{splitEn2},
		Lemmas \ref{treatEnprime}, \ref{treatEnst} and
		\ref{treatEnstst}.
	\end{proof}

\end{document}